%% file: main.tex
\pdfoutput=1
\documentclass{amsart}

\newcommand\tab[1][0.5cm]{\hspace*{#1}}
\title[Effective bounds for Bloch-Kato Selmer groups]{Refined effective bounds for Bloch-Kato Selmer groups associated to hyperelliptic curves}
\author{Lee Berry}
\address{Department of Mathematics, King’s College London, Strand, London, WC2R 2LS, UK}
\email{lee.berry@kcl.ac.uk}

\date\today
\input{header}
\begin{document}

\begin{abstract}
We develop refined methods to effectively bound the dimension of Bloch-Kato Selmer groups associated to the higher Chow group $\mathrm{CH}^2(J,1)$, where $J$ is the Jacobian of a hyperelliptic curve $X$. This extends the recent work of Dogra on explicit $2$-descent for these Selmer groups to include cases where $X$ does not have a rational Weierstrass point. Additionally, we develop methods for obtaining sharper dimension bounds under the assumption that $X$ has good ordinary reduction at $2$. As a consequence, we establish new criteria for deducing finiteness of the depth $2$ Chabauty-Kim set $X(\mathbb{Q}_2)_2$, and demonstrate the efficacy of these criteria on curves from the LMFDB.
\end{abstract}
\maketitle

\tableofcontents

\section{Introduction}

Fix a prime $p$ and denote by $V$ a finite-dimensional $\mathbb{Q}_p$-vector space with a continuous action of $\mathrm{Gal}(\overline{\mathbb{Q}}/\mathbb{Q})$. In \cite{bloch1990functions}, Bloch and Kato introduced a subspace $H^1_f(\mathbb{Q},V)$
of the Galois cohomology group $H^1(\mathbb{Q},V)$, which is finite-dimensional over $\mathbb{Q}_p$ and generalises the classical notion of a Selmer group. Specifically, if $A$ is an abelian variety over $\mathbb{Q}$, then the $p^\infty$-Selmer group of $A$, upon tensoring with $\mathbb{Q}_p$, naturally identifies with $H^1_f(\mathbb{Q},V_pA)$, where $V=V_pA$ denotes the $\mathbb{Q}_p$-rational Tate module of $A$. Accordingly, the image of the Kummer map 
\begin{align*}
A(\mathbb{Q})\otimes \mathbb{Q}_p \rightarrow H^1(\mathbb{Q},V_pA)
\end{align*}
is contained in $H^1_f(\mathbb{Q},V_pA)$ and surjects onto this Selmer group under finiteness assumptions on the Tate-Shafarevich group $\Sha(A)$ of $A$. A natural generalisation of this map for more general Galois representations arising from \'{e}tale cohomology is the \'{e}tale regulator 
\begin{align}\label{ChowKummer}
\mathrm{CH}^i(A,j)\otimes \mathbb{Q}_p \rightarrow H^1\big(\mathbb{Q},H^{2i-j-1}_{\textrm{\'{e}t}}(A_{\overline{\mathbb{Q}}},\mathbb{Q}_p(i))\big),
\end{align}
where $\mathrm{CH}^i(A,j)$ denotes a higher Chow group of $A$, as defined in \cite{bloch1986algebraic}. A special case of the Bloch-Kato conjectures \cite[Conjecture 5.3(i)]{bloch1990functions} asserts that \eqref{ChowKummer} defines an isomorphism onto the corresponding Selmer group. We do not attempt to survey the scarcity of known cases of the conjecture; suffice it to say, even effective determination of dimension bounds for the relevant Bloch-Kato Selmer group is highly nontrivial. In this paper, we focus on the case where $i=2$, $j=1$, and $A=J$ is the Jacobian of a hyperelliptic curve $X$, and develop effective methods for bounding the dimension of the corresponding Bloch-Kato Selmer group $H^1_f(\mathbb{Q},\wedge^2 V_2J)$. Moreover, we develop methods to obtain improved bounds whenever $X$ has good ordinary reduction at $2$. As the main application of these unconditional bounds, we consider their utility in determining rational points on hyperelliptic curves via the effective Chabauty-Kim method. 

Let $X$ be a smooth projective curve over $\mathbb{Q}$ of genus $g\geq 2$. By Faltings' theorem, which resolved the longstanding Mordell conjecture, the set of rational points $X(\mathbb{Q})$ is finite. However, the question of whether an algorithm exists to compute this finite set of points remains wide open in general. Fortunately, the Chabauty-Kim method, as developed in \cite{kim2005motivic,kim2009unipotent}, has provided insight into addressing this question for an expanding class of higher genus curves. Roughly speaking, the Chabauty-Kim method constructs for each integer $n>0$ a set of $p$-adic points $X(\mathbb{Q}_p)_n$, which contains $X(\mathbb{Q})$ and forms a nested sequence of inclusions
\begin{align*}
X(\mathbb{Q})\subseteq\cdots \subseteq X(\mathbb{Q}_p)_n \subseteq\cdots \subseteq X(\mathbb{Q}_p)_2  \subseteq X(\mathbb{Q}_p)_1 \subseteq X(\mathbb{Q}_p).
\end{align*}
The Bloch-Kato conjectures, among other well-known conjectures, imply that for $n$ sufficiently large, the \textit{depth $n$ Chabauty-Kim set} $X(\mathbb{Q}_p)_n$ is finite (see \cite[§3]{kim2009unipotent}). Therefore, a potential general strategy for determining $X(\mathbb{Q})$ via Chabauty-Kim theory could proceed as follows.
\begin{enumerate}\label{Pt1}
    \item[(1)] Establish criteria for the finiteness of the depth $n$ Chabauty-Kim set $X(\mathbb{Q}_p)_n$. 
    \item[(2)] Once $X(\mathbb{Q}_p)_n$ is shown to be finite for some $n>0$, compute $X(\mathbb{Q}_p)_n$ explicitly. 
    \item[(3)] Identify and discard the points in $X(\mathbb{Q}_p)_n$ that are not $\mathbb{Q}$-rational. 
\end{enumerate}    
We do not elaborate on the final point here; for a given curve, it is expected that in most cases this could be dealt with by an application of the Mordell-Weil sieve of \cite{Bruin_Stoll_2010}.

At depth $1$, the first two steps of the strategy are addressed by the effective method of Chabauty-Coleman \cite{Chab,colemanint}. This method guarantees that $X(\mathbb{Q}_p)_1$ is finite whenever the rank of the Mordell-Weil group $J(\mathbb{Q})$ is less than the genus $g$ of $X$; in this case, explicit computation is theoretically possible via Coleman integration. At depth $2$, the method of quadratic Chabauty for rational points, introduced in \cite{Balakrishnan_2018}, asserts that $X(\mathbb{Q}_p)_2$ is finite whenever $\mathrm{rk}\hspace{0.8mm} J(\mathbb{Q})<g+\mathrm{rk}\hspace{0.8mm}\mathrm{NS}(J)-1$, where $\mathrm{NS}(J)$ denotes the N\'{e}ron-Severi group of the Jacobian of $X$. In this case, the method can effectively compute a finite intermediate subset -- between the depth $1$ and $2$ sets -- using the machinery of $p$-adic heights. 

Given that $\mathrm{NS}(J)$ has nonzero rank, the quadratic Chabauty method unsurprisingly has broader applicability than the Chabauty-Coleman method. For example, one advantage of quadratic Chabauty is its success at effective determination of rational points on modular curves  (see, for example, \cite{Balakrishnan_Dogra_Müller_Tuitman_Vonk_2023}). However, the quadratic Chabauty criterion remains restrictive in general. For instance, \cite{bk2} notes that genus $2$, rank $2$ odd hyperelliptic curves on the LMFDB \cite{lmfdb} rarely satisfy the criterion. The obstruction to obtaining a more widely applicable criterion in depth $2$ is the Bloch-Kato Selmer group $H^1_f(\mathbb{Q},\wedge^2 V_pJ)$, as demonstrated by the following case of a general theorem of Kim.

\begin{nonumtheorem}[\texorpdfstring{\cite[Theorem 1]{kim2009unipotent}}{[Kim09, Theorem 1]}, \texorpdfstring{\cite[Lemma 25]{bk1}}{[Dog23, Lemma 25]}]
Suppose 
\begin{align*}
\dim_{\mathbb{Q}_p} H_f^1(\mathbb{Q},\wedge^2 V_pJ) < \frac{1}{2}(3g-2)(g+1)-\mathrm{rk}\hspace{0.8mm} J(\mathbb{Q}).
\end{align*}
Then $X(\mathbb{Q}_p)_2$ is finite. 
\end{nonumtheorem}

Consequently, the $p=2$ case of this theorem, coupled with the aforementioned effective methods developed in this paper, yield refined finiteness criteria for the depth $2$ set $X(\mathbb{Q}_2)_2$ for hyperelliptic curves. As alluded to above, these criteria are particularly applicable for curves with good ordinary reduction at $2$. In the above framework, this paper focuses predominantly on addressing point \eqref{Pt1} for such curves.

\subsection{$2$-descent for Bloch-Kato Selmer groups}
To bound the $\mathbb{Q}_2$-dimension of $H^1_f(\mathbb{Q},\wedge^2V_2J)$, it suffices to bound the rank of $H^1_f(\mathbb{Q},\wedge^2 T_2J)$, where $T_2J$ denotes the $2$-adic Tate module of $J$. Since the mod-$2$ quotient of $\wedge^2 T_2J$ is isomorphic to $\wedge^2 J[2]$, the strategy is to bound this rank by computing subgroups of $H^1(\mathbb{Q},\wedge^2 J[2])$ that contain the image of $H^1_f(\mathbb{Q},\wedge^2 T_2J)$. To elaborate on this strategy further, we introduce some notation. 

Suppose $f\in K[x]$ is a polynomial with coefficients in a field $K$ of characteristic different from $2$. Then define the following \'{e}tale algebras: 
\begin{align}\label{EtAlg}
\begin{split}
K_f&\coloneqq K[x]/(f(x)) \\
K_{f,2}&\coloneqq K[x,y,\tfrac{1}{x-y}]/\big(f(x),f(y)\big),
\end{split}
\end{align}
and denote by $K_f^{(2)}$ the subalgebra of $K_{f,2}$ fixed by the endomorphism that interchanges the classes of $x$ and $y$ in $K_{f,2}$. If $X$ is a hyperelliptic curve over $K$ with exactly one rational Weierstrass point, so that it can be represented by the affine model $y^2=f(x)$ for some odd degree $f\in K[x]$, then \cite[Proposition 1]{bk1} establishes the description
\begin{align}\label{IntroDesc}
H^1(K,\wedge^2 J[2])\simeq \mathrm{Ker}\big(K^{(2),\times}_f\otimes \mathbb{F}_2 \xrightarrow{\text{Nm}} K_f^\times \otimes \mathbb{F}_2\big).
\end{align}
In Lemma \ref{WedgeSqJm} below, we establish an analogous description involving a generalised Jacobian of $X$, which provides a rudimentary upper bound for curves with no rational Weierstrass point. While these descriptions alone can provide finiteness criteria applicable to curves of high genus, they are too coarse to deduce finiteness of $X(\mathbb{Q}_2)_2$ for genus $2$ curves, for instance. The above mentioned \cite{bk1,bk2} provide two approaches to refining these upper bounds, while the current paper considers a third. Although most effectively applicable to hyperelliptic curves with good ordinary reduction at $2$, the advantages of the methods in this paper over those previously mentioned are twofold. Firstly, checking a given curve satisfies the finiteness criteria of this paper is significantly more computationally efficient; secondly, the current methods also apply to hyperelliptic curves with no rational Weierstrass point. 

To state the main results, suppose $X$ is a hyperelliptic curve over $K=\mathbb{Q}$ with good ordinary reduction at $2$, given by the affine model $y^2=f(x)$, where $f$ is not necessarily of odd degree. Suppose $S$ is a finite set of primes, and denote
\begin{align*}
\mathcal{A}(f,S)\coloneqq \mathrm{Ker}\Big(H^1\big(\mathrm{Gal}(\mathbb{Q}_S/\mathbb{Q}),\mathrm{Ind}_K^{K_f^{(2)}}\mathbb{F}_2\big)\rightarrow H^1\big(\mathrm{Gal}(\mathbb{Q}_S/\mathbb{Q}),\mathrm{Ind}_K^{K_f}\mathbb{F}_2\big)\Big),
\end{align*}
as the kernel of the map induced by the norm map in \eqref{IntroDesc}, where $\mathbb{Q}_S$ denotes the maximal extension of $\mathbb{Q}$ unramified outside of $S$. In particular, if the degree of $f$ is odd then $\mathcal{A}(f,S)$ is precisely the cohomology group $H^1(\mathrm{Gal}(\mathbb{Q}_S/\mathbb{Q}),\wedge^2J[2])$, where $J$ is the Jacobian of the curve defined by $f$. In Lemma \ref{ExplicitMap} and Definition \ref{DefThetaDr} below, we define an effectively computable map 
\begin{align*}
\theta_{\mathrm{dR}}: \mathcal{A}(f,S)\rightarrow H^1(\mathbb{Q}_2^{\mathrm{ord}},\mathbb{F}_2)/H^1_{\mathrm{nr}}(\mathbb{Q}_2^{\mathrm{ord}},\mathbb{F}_2),
\end{align*}
where $\mathbb{Q}_2^{\mathrm{ord}}$ denotes the \'{e}tale algebra $(\mathbb{Q}_2)_Q^{(2)}$ for $y^2+Q(x)y=P(x)$ a suitably chosen Weierstrass model of $X$ over $\mathbb{Q}$; this choice is detailed in \eqref{Q2ord} below. Although its definition certainly depends on the set $S$, the notation $\theta_{\mathrm{dR}}$ should not be ambiguous as $S$ will mostly remain unchanged throughout the text. Roughly speaking, the kernel of $\theta_{\mathrm{dR}}$ approximates the image of $H^1_f(\mathbb{Q},\wedge^2 T_2J)$ in $H^1(\mathbb{Q},\wedge^2 J[2])$ and allows for the following dimension bounds, distinguished by the number of rational Weierstrass points (RWPs) on $X$. 

\begin{introtheorem}[Theorems \ref{MainOddTheorem} and \ref{MainEvenThm}]\label{ThmA}
Suppose $X$ is a genus $g$ curve as above, and let $S$ in the definition of $\theta_{\mathrm{dR}}$ be the set of primes that contains $2$ and the primes where $X$ does not have semistable reduction. Then 
\begin{align*}
\dim_{\mathbb{Q}_2}H^1_f(\mathbb{Q},\wedge^2 V_2 J) \leq \begin{cases}
\dim_{\mathbb{F}_2}\mathrm{Ker}(\theta_{\mathrm{dR}})-2 , &\text{if $X$ has one RWP}; \\
\dim_{\mathbb{F}_2}\mathrm{Ker}(\theta_{\mathrm{dR}})-g-1, &\text{if $X$ has no RWPs}.
\end{cases}
\end{align*}
\end{introtheorem}

Given the theorem of Kim above, this facilitates the desired finiteness criteria. In the case that $X$ has one RWP, the methods of \cite{bk2} can be used in tandem with the approximation of the image of $H^1_f(\mathbb{Q},\wedge^2 T_2J)$ in $H^1(\mathbb{Q},\wedge^2 J[2])$ afforded by $\theta_{\mathrm{dR}}$; the compatibility of the two methods is guaranteed by the proof of Theorem \ref{MainEvenThm} below. In particular, there are curves for which finiteness of $X(\mathbb{Q}_2)_2$ can be concluded by a combination of the two methods, but not by either of the methods alone; see the discussion after Theorem \ref{ExampleComputation} for an example of such a curve. 

The computational efficiency mentioned above allows for the effective computation of $\mathrm{Ker}(\theta_{\mathrm{dR}})$ for the appropriate genus $2$ curves on the LMFDB. The following is an illustration of the efficacy of the resulting refined quadratic Chabauty criteria for these curves.

\begin{introtheorem}[Theorem \ref{ExampleComputation}]\label{ThmB}
Of the $1,138$ genus $2$ hyperelliptic curves on the LMFDB with Mordell-Weil rank $2$, exactly one RWP, and good ordinary reduction at $2$, finiteness of $X(\mathbb{Q}_2)_2$ is guaranteed for at least $574$ curves. 
\end{introtheorem}

In contrast, the N\'{e}ron-Severi group of the Jacobian of each of these $1,138$ curves has rank $1$, so none satisfy the classical quadratic Chabauty criterion discussed above. We are also able to verify finiteness of $X(\mathbb{Q}_2)_2$ via Theorem \ref{ThmA} for several genus $3$ curves -- including those without a RWP -- which we include in the final section.

The structure of the paper is as follows. In Section \ref{SectionPrereqs}, we discuss general Galois cohomological prerequisites necessary for defining $\theta_{\mathrm{dR}}$. We also recall the proof of \eqref{IntroDesc} and prove the analogue for the case that $X$ does not have a RWP. A brief outline of the Chabauty-Kim method is included in Section \ref{SectionChabauty}, as well as further discussion on the utility of $2$-descent for Bloch-Kato Selmer groups. Additionally, in this section we relate $\theta_{\mathrm{dR}}$ to the desired Bloch-Kato Selmer groups. In Section \ref{SectionCriteria}, we show that the main results follow straightforwardly from the general theory outlined in the previous sections. Finally, Section \ref{SectionExamples} provides the outline of a general algorithm to compute $\mathrm{Ker}(\theta_{\mathrm{dR}})$, and presents worked examples for low genus curves. Moreover, it includes the various statistics for relevant curves on the LMFDB, as alluded to in Theorem \ref{ThmB}.

\subsection{Notation and conventions}
Suppose $K$ is a field. Fix $K^{\mathrm{sep}}$ a separable closure of $K$ and denote by $G_K$ the absolute Galois group $\mathrm{Gal}(K^{\mathrm{sep}}/K)$. If $S$ is a finite set of primes of $K$, then denote by $G_{K,S}$ the Galois group of the maximal extension of $K$ unramified outside of $S$; in the case that $K=\mathbb{Q}$, denote $G_S\coloneqq G_{\mathbb{Q},S}$. 

For a finite extension $M/L$ of finite \'{e}tale $K$-algebras, denote their decomposition into a product of fields as
\begin{align*}
M=\prod_{i,j} M_{ij}, \tab L=\prod_{i} L_i,
\end{align*}
such that $M_{ij}$ is a finite field extension of $L_i$. In this case, if $A=(A_i)_i$ is a tuple of $G_{L_i}$-modules, then define the tuple of $G_{M_{ij}}$-modules 
\begin{align*}
\mathrm{Ind}_L^M A \coloneqq \Big(\mathrm{Ind}_{L_i}^{M_{ij}}A_i\Big)_{i,j}=\Big(\mathrm{Ind}_{G_{M_{ij}}}^{G_{L_i}}A_i\Big)_{i,j}.
\end{align*}
Throughout the paper, we exploit the identification of $\mathrm{Ind}_K^{K_f}\mathbb{F}_2$ with the $\mathbb{F}_2$-vector space generated by the roots of $f$, where $K_f$ is the \'{e}tale algebra defined in \eqref{EtAlg}. 

If $K$ is a number field, $V$ is a $\mathbb{Q}_p$-vector space with a continuous action of $G_K$, and $v$ is a prime of $K$ above $p$, we follow \cite{bloch1990functions} in defining the subspaces $H^1_e(K_v,V)\subseteq H^1_f(K_v,V)\subseteq H^1_g(K_v,V) \subseteq H^1(K_v,V)$ given by 
\begin{align*}
H^1_e(K_v,V)&\coloneqq \mathrm{Ker}\big(H^1(K_v,V)\rightarrow H^1(K_v,V\otimes B_{\mathrm{cris}}^{\varphi =1}) \big)\\
H^1_f(K_v,V)&\coloneqq \mathrm{Ker}\big(H^1(K_v,V)\rightarrow H^1(K_v,V\otimes B_{\mathrm{cris}})\big)\\
H^1_g(K_v,V)&\coloneqq \mathrm{Ker}\big(H^1(K_v,V)\rightarrow H^1(K_v,V\otimes B_{\mathrm{dR}})\big). 
\end{align*}
For $v$ not lying above $p$, we adopt the convention 
\begin{align*}
H^1_f(K_v,V)\coloneqq H^1_{\mathrm{nr}}(K_v,V), \tab H^1_g(K_v,V)\coloneqq H^1(K_v,V).
\end{align*}
The global analogue $H^1_*(K,V)$ is then defined as the subgroup of $H^1(K,V)$ consisting of elements that, when restricted to the decomposition groups $G_v$, lie in $H^1_*(K_v,V)$ for each prime $v$ of $K$. For $*\in \{\mathrm{dR},\mathrm{cris}\}$, define
\begin{align*}
D_*(V)\coloneqq H^0(K,V\otimes B_*),
\end{align*}
where $D_{\mathrm{cris}}(V)$ is equipped with a Frobenius, denoted $\varphi$ throughout, induced by the Frobenius on $B_{\mathrm{cris}}$, and $D_{\mathrm{dR}}(V)$ is endowed with a separated and exhaustive decreasing filtration $\mathrm{Fil^\bullet D_{\mathrm{dR}}(V)}$.

Finally, for an abelian group $N$ denote
\begin{align*}
\wedge^2 N \coloneqq N^{\otimes 2}\big / \langle x  \otimes x : x\in N \rangle .
\end{align*}
Throughout this paper, $N$ will be an $\mathbb{F}_2$-vector space and, without specifying, we will identify $\wedge^2 N$ with the $\mathbb{F}_2$-vector space generated by unordered pairs of distinct basis elements of $N$.

\subsection{Acknowledgements} 
I am greatly indebted to my supervisor, Netan Dogra, for suggesting this project, for proofreading earlier drafts, and for numerous insightful conversations on the mathematics of this paper and mathematics more broadly. This work was supported by the Royal Society [RF\textbackslash ERE\textbackslash 210020] and a PhD studentship from the Faculty of Natural, Mathematical and Engineering Sciences at King's College London.

\section{Preliminaries in Galois cohomology}\label{SectionPrereqs}
We first recall key results regarding the Galois cohomology of the wedge square, as detailed in \cite{bk1,bk2}. Suppose $f\in K[x]$ is a polynomial with coefficients in a field $K$ of characteristic different from $2$, and denote by $K_f$ and $K_f^{(2)}$ the \'{e}tale algebras defined in \eqref{EtAlg}. Given the importance of its proof in the subsequent material, we include the following previously mentioned description of $H^1(K,\wedge^2 J[2])$ and its derivation. 

\begin{lemma}\label{BK1Lemmas}
\begin{itemize}
    \item[(1)] There exists an isomorphism of finite $G_K$-modules
    \begin{align*}
    \wedge^2\mathrm{Ind}_K^{K_f}\mathbb{F}_2 \simeq \mathrm{Ind}_K^{K_f^{(2)}}\mathbb{F}_2.
    \end{align*}
    \item[(2)] Suppose that $X$ is a hyperelliptic curve over $K$ given by the model $y^2=f(x)$, where the degree of $f$ is odd. Then 
    \begin{align}\label{CohomOddWedge}
    \begin{split}
    H^1(K,\wedge^2J[2])&\simeq \mathrm{Ker} \Big(H^1\big(K,\mathrm{Ind}_K^{K_f^{(2)}}\mathbb{F}_2\big)\rightarrow H^1\big(K,\mathrm{Ind}_K^{K_f}\mathbb{F}_2\big)\Big)\\
    &\simeq \mathrm{Ker}\big(K_f^{(2),\times}\otimes \mathbb{F}_2 \xrightarrow{\mathrm{Nm}} K_f^\times\otimes \mathbb{F}_2\big).
    \end{split}
    \end{align}
\end{itemize}
\end{lemma}

\begin{proof}
Part (1) follows from the first two lemmas of \cite{bk2}, and part (2) is established by its first proposition; we briefly outline the proof of the latter. Recall that when $X$ is given by an odd degree model, there exists the following split exact sequence (see, for example, the proof of \cite[Theorem 1.1]{SCHAEFER1995219}):
\begin{align}\label{DefinitionOddJ}
0\longrightarrow J[2] \longrightarrow \mathrm{Ind}_K^{K_f}\mathbb{F}_2 \xlongrightarrow{\mathrm{Nm}} \mathbb{F}_2 \longrightarrow 0. 
\end{align}
Upon taking exterior powers, this sequence induces the short exact sequence 
\begin{align}\label{DefinitionOddWedge}
0\longrightarrow \wedge^2 J[2] \longrightarrow \wedge^2 \mathrm{Ind}_K^{K_f}\mathbb{F}_2 \longrightarrow J[2] \longrightarrow 0,
\end{align}
obtained from the map that sends $\alpha \wedge \beta $ in $\wedge^2 \mathrm{Ind}_K^{K_f}\mathbb{F}_2$ to $\alpha + \beta $ in $J[2]$, where $\alpha,\beta$ are distinct roots of $f$. This sequence remains split given the section $J[2]\rightarrow \wedge^2 \mathrm{Ind}_K^{K_f}\mathbb{F}_2$ induced by the map
\begin{align*}
\mathrm{Ind}_K^{K_f}\mathbb{F}_2 &\longrightarrow \wedge^2 \mathrm{Ind}_K^{K_f}\mathbb{F}_2 \\
\alpha &\longmapsto \sum_{\beta \neq \alpha} \alpha \wedge \beta. \nonumber
\end{align*}
Consequently, the first isomorphism follows from part (1). The final isomorphism is a consequence of Shapiro's lemma, Kummer theory, and the commutativity of the square
\begin{center}
\begin{tikzcd}
H^1\big(K,\mathrm{Ind}_K^{K_f^{(2)}}\mathbb{F}_2\big)\arrow{r} \arrow{d} & H^1(K,\mathrm{Ind}_K^{K_f}\mathbb{F}_2) \arrow{d} \\
K_f^{(2),\times}\otimes \mathbb{F}_2 \arrow{r} & K_f^\times\otimes \mathbb{F}_2,
\end{tikzcd}
\end{center}
where the bottom map factors through the norm map $K_{f,2}^\times \otimes \mathbb{F}_2 \rightarrow K_f^\times \otimes \mathbb{F}_2$; this is proved in \cite[Lemma 6]{bk1}. 
\end{proof}

\subsection{The even degree case and the generalised Jacobian}
Suppose in this subsection that $X$ is a hyperelliptic curve over $K$ with no rational Weierstrass point. It can therefore be given by the affine model $y^2=f(x)$, where the degree of $f$ is necessarily even. As in the case of explicit $2$-descent for Jacobians of such curves, the difficulty in describing an analogue of \eqref{CohomOddWedge} is that $J[2]$ can no longer be realised as a submodule of $\mathrm{Ind}_K^{K_f}\mathbb{F}_2$. To circumvent this difficulty, it is necessary to introduce a \textit{generalised Jacobian} associated to $X$, following the strategy in \cite{poonsch}. 

Specifically, denote by $J_\mathfrak{m}$ the generalised Jacobian of $X$ relative to the modulus $\mathfrak{m}=\infty^++\infty^-$, where $\infty ^+,\infty^-$ denote the two points at infinity of $X$. The precise definition of the generalised Jacobian is not required in the material that follows; its definition and fundamental properties can be found in \cite{serre2012algebraic}, for instance. It suffices to note that in the case that $X$ is hyperelliptic, $J_\mathfrak{m}$ is a commutative algebraic group that fits into the following short exact sequence of semiabelian varieties: 
\begin{align}\label{GenJDefSES}
0\longrightarrow \mathbb{G}_m(\chi_c) \longrightarrow J_\mathfrak{m} \longrightarrow J \longrightarrow 0,
\end{align}
where $\mathbb{G}_m(\chi_c)$ is the twist of $\mathbb{G}_m$ by the character $\chi_c$ associated to the leading coefficient $c \in K^\times$ of $f$ (see \cite[§4]{poonsch}). 

Of particular importance to explicit descent is that there exists an analogue of the short exact sequence \eqref{DefinitionOddJ}, given by
\begin{align}\label{DefinitionEvenJm}
0\longrightarrow J_\mathfrak{m}[2] \longrightarrow \mathrm{Ind}_K^{K_f}\mathbb{F}_2 \xlongrightarrow{\mathrm{Nm}} \mathbb{F}_2 \longrightarrow 0,  
\end{align}
in the case that $f$ is of even degree; see \cite[Proposition 6.2]{poonsch} for its derivation. Unfortunately, this sequence is no longer split, as is also true of its analogue to \eqref{DefinitionOddWedge}. Nevertheless, the following lemma provides an even degree version of Lemma \ref{BK1Lemmas}.

\begin{lemma}\label{WedgeSqJm}
Suppose $X$ is a hyperelliptic curve given by the affine model above, where the degree of $f\in K[x]$ is even. Then we have the isomorphism 
\begin{align*}
H^1\big(K,\wedge^2 J_\mathfrak{m}[2]\big)\simeq \mathrm{Ker}\Big(H^1\big(K,\mathrm{Ind}_K^{K_f^{(2)}}\mathbb{F}_2\big)\rightarrow H^1\big(K,J_\mathfrak{m}[2]\big)\Big).
\end{align*}
\end{lemma}

\begin{proof}
Identically following the proof of part (2) of Lemma \ref{BK1Lemmas}, the short exact sequence \eqref{DefinitionEvenJm} induces the exact sequence
\begin{align}\label{DefinitionEvenWedge}
\begin{split}
0\longrightarrow \wedge^2 J_\mathfrak{m}[2] \longrightarrow \wedge^2 \mathrm{Ind}_K^{K_f}\mathbb{F}_2 &\longrightarrow J_\mathfrak{m}[2] \longrightarrow 0 \\
\alpha\wedge \beta &\longmapsto \alpha +\beta. 
\end{split}
\end{align}
Any element of $H^0(K,J_\mathfrak{m}[2])$ can be expressed as a sum $\sum_{i=1}^{2d}\alpha_i$ of roots $\alpha_i$ of $f$, such that $\{\alpha_i\}_{i=1}^{2d}$ is precisely the roots of the product of irreducible factors of $f$ in $K[x]$. However, each such element is the image of 
\begin{align*}
    \sum_{i=1}^{2d-1}\sum_{j=i+1}^{2d}(\alpha_i\wedge \alpha_j) \in H^0\big(K,\wedge^2 \mathrm{Ind}_{K}^{K_f}\mathbb{F}_2\big).
\end{align*} Consequently, the boundary map $H^0(K,J_\mathfrak{m}[2]) \rightarrow H^1\big(K,\wedge^2 J_\mathfrak{m}[2]\big)$ associated to the short exact sequence \eqref{DefinitionEvenWedge} is trivial. The result then follows from part (1) of Lemma \ref{BK1Lemmas}.
\end{proof}

\begin{remark}
In the case that $f$ has an irreducible factor of odd degree in $K[x]$, this lemma recovers an isomorphism equivalent to \eqref{CohomOddWedge}. Specifically, we obtain that
\begin{align*}
H^1(K,\wedge^2J_\mathfrak{m}[2])\simeq \mathrm{Ker} \Big(H^1\big(K,\mathrm{Ind}_K^{K_f^{(2)}}\mathbb{F}_2\big)\rightarrow H^1\big(K,\mathrm{Ind}_K^{K_f}\mathbb{F}_2\big)\Big),
\end{align*}
which is more amenable to field-theoretic description (via Shapiro's lemma and Kummer theory) than the right-hand side of the isomorphism in the statement of the lemma. 
\end{remark}

\subsection{Hyperelliptic curves over $\mathbb{Q}_2$ with good ordinary reduction}

Throughout this subsection, unless otherwise specified, fix $K=\mathbb{Q}_2$. Suppose $X$ is a genus $g$ hyperelliptic curve over $K$ given by the Weierstrass model
\begin{align}\label{WeierstrassModel}
y^2 + Q(x)y = P(x),
\end{align}
such that $\deg f\in \{2g+1,2g+2\}$, where $f\coloneqq 4P+Q^2$. In this subsection, we impose no restriction on the number of RWPs on $X$. Additionally, assume that $X$ has good reduction, so that we have the reduction map on $2$-torsion
\begin{align}\label{ReductionMap}
J[2] \longrightarrow J_{\mathbb{F}_2}[2],
\end{align}
where $J_{\mathbb{F}_2}$ denotes the special fibre of $J$. Moreover, assume that $X$ has ordinary reduction; this implies that $X$ can be given by the Weierstrass model \eqref{WeierstrassModel} such that the reduction $\overline{Q}\in\mathbb{F}_2[x]$ defines a separable, degree $g+1$ polynomial (see \cite[Remark 2.2]{DOKCHITSER2023264}). In this case, $J_{\mathbb{F}_2}[2]$ has a description analogous to \eqref{DefinitionOddJ}, provided by the following general lemma. 

\begin{lemma}[\texorpdfstring{\cite[Lemma 2.1]{DOKCHITSER2023264}}{[DM23, Lemma 2.1]}]
Let $X$ be a hyperelliptic curve of genus $g\geq 2$ over a field of characteristic $2$, given by the Weierstrass model \eqref{WeierstrassModel} such that $\deg Q = g+1$. Let $\mathcal{S}$ denote the set of subsets of roots of $Q$ of even size, endowed with a group structure defined by symmetric difference. The map $\mathcal{S}\rightarrow J[2]$ defined by
\begin{align*}
A\mapsto \sum_{\beta\in \mathrm{Roots}(Q)} (\beta,\sqrt{P(\beta)}) -\frac{|A|}{2}(\infty^++\infty ^-)
\end{align*}
defines an isomorphism of Galois modules. 
\end{lemma}

In particular, given that the action of $G_{\mathbb{Q}_2}$ on the roots of $\overline{Q}\in \mathbb{F}_2[x]$ factors through the action of $G_{\mathbb{F}_2}$, we obtain, from the above lemma, the exact sequence 
\begin{align}\label{SpFibSES}
0\longrightarrow J_{\mathbb{F}_2}[2] \longrightarrow \mathrm{Ind}_{K}^{K_{Q}}\mathbb{F}_2 \xlongrightarrow{\mathrm{Nm}} \mathbb{F}_2 \longrightarrow 0,
\end{align}
from which we deduce that \eqref{ReductionMap} can be explicitly described via the roots of the polynomials $f$ and $Q$. To obtain such a description, we require the following notation, which shall be used throughout (cf. Notation 1.1 and the proof of \cite[Lemma 3.1]{DOKCHITSER2023264}). 

\begin{notation}\label{RootsNotation}
Suppose $X$ is a hyperelliptic curve over $\mathbb{Q}_2$ with good ordinary reduction, given by the Weierstrass equation \eqref{WeierstrassModel} such that $\deg Q = g+1$. Given that $\overline{Q}\in \mathbb{F}_2[x]$ is separable, the roots $\beta_i$ of $Q$ are contained in the maximal unramified extension $\mathbb{Q}_2^{nr}$ of $\mathbb{Q}_2$. As above, setting $f=4P+Q^2$, we have 
\begin{align*}
f(x)\equiv Q(x)^2 \equiv \overline{c}\cdot\prod_{i=1}^{g+1}(x-\beta_i)^2\textrm{ (mod }4),
\end{align*}
which lifts, via Hensel's lemma, to the factorisation in $\mathbb{Q}_2^{nr}[x]$ given by 
\begin{align*}
f(x)=c\cdot\prod_{i=1}^{g+1}q_i(x), 
\end{align*}
for quadratic $q_i$ such that $q_i(x) \equiv (x-\beta_i)^2 \textrm{ (mod }4)$. Denote the roots of $q_i$ as $\alpha_{i,1}$ and $\alpha_{i,2}$ so that the reductions of $\alpha_{i,1}$, $\alpha_{i,2}$ and $\beta_i$ coincide in $\overline{\mathbb{F}}_2$. 
\end{notation}

We will follow the convention of labeling the roots of $f$ and $Q$ in accordance with Notation \ref{RootsNotation} for the remainder of the paper. In particular, it is clear that the reduction map \eqref{ReductionMap} is induced by the explicit map
\begin{align}\label{IndRed}
\mathrm{Ind}_K^{K_f}\mathbb{F}_2 &\longrightarrow \mathrm{Ind}_K^{K_Q}\mathbb{F}_2, 
\end{align}
which sends a root $\alpha_{i,j}$ of $f$ to the corresponding root $\beta_i$ of $Q$. 

Throughout this paper, we will be interested in the map -- and the corresponding maps on the associated cohomology groups -- given by
\begin{align}\label{ReductionWedgeMap}
\wedge^2 J[2] \longrightarrow \wedge^2 J_{\mathbb{F}_2}[2], 
\end{align}
induced by the reduction map \eqref{ReductionMap}. Towards this, note that an identical argument to the proof of Lemma \ref{BK1Lemmas}(2), applied to the exact sequence \eqref{SpFibSES}, yields the isomorphism 
\begin{align}\label{OrdinaryWedge}
H^1(K,\wedge^2 J_{\mathbb{F}_2}[2])\simeq \mathrm{Ker}\big(K_{Q}^{(2),\times} \otimes \mathbb{F}_2 \rightarrow H^1(K,J_{\mathbb{F}_2}[2]) \big).
\end{align} 
In the case that $g$ is even, the sequence \eqref{SpFibSES} splits and therefore the right-hand side of \eqref{OrdinaryWedge} is just the kernel of the map to $K_Q^\times \otimes \mathbb{F}_2$. Essential in the material that follows -- particularly in defining the map $\theta_{\mathrm{dR}}$ mentioned in the introduction -- is an explicit description of
\begin{align*}
H^1(K,\wedge^2J[2])\longrightarrow H^1(K,\wedge^2 J_{\mathbb{F}_2}[2]),
\end{align*}
and related maps, associated to \eqref{ReductionWedgeMap}. Given the field theoretic descriptions of $H^1(K,\wedge^2J[2])$, $H^1(K,\wedge^2J_\mathfrak{m}[2])$ and $H^1(K,\wedge^2J_{\mathbb{F}_2}[2])$ afforded by Lemmas \ref{BK1Lemmas} and \ref{WedgeSqJm}, and the isomorphism \eqref{OrdinaryWedge}, respectively, it will therefore be necessary to have an explicit description of the map 
\begin{align}\label{FieldTheoreticMap}
K_f^{(2),\times}\otimes \mathbb{F}_2 \longrightarrow K_Q^{(2),\times}\otimes \mathbb{F}_2,
\end{align}
obtained from the map on finite $G_K$-modules 
\begin{align}\label{IndReductionWedgeMap} 
\begin{split}
\wedge^2 \mathrm{Ind}_K^{K_f}\mathbb{F}_2 &\longrightarrow \wedge^2\mathrm{Ind}_K^{K_Q}\mathbb{F}_2\\
\alpha_{i,j}\wedge \alpha_{k,l} &\longmapsto \beta_i\wedge \beta_k. 
\end{split}
\end{align}
The following general lemma is essential for this purpose, where we do not necessarily require that $K=\mathbb{Q}_2$. 

\begin{lemma}\label{GeneralInd}
Suppose $K$ is a field of characteristic different from $2$ and $L\subseteq M$ are finite extensions of $K$. Then the map
\begin{align}\label{GeneralMap}
\begin{split}
\mathrm{Ind}_K^M \mathbb{F}_2 & \longrightarrow \mathrm{Ind}_K^L \mathbb{F}_2\\ 
\sigma \otimes_{\mathbb{Z}[G_M]} 1 &\longmapsto \sigma \otimes_{\mathbb{Z}[G_L]} 1 
\end{split}
\end{align}
induces the commutative diagram 
\begin{center}
\begin{tikzcd}
H^1(K,\mathrm{Ind}_K^M \mathbb{F}_2) \arrow{r}\arrow{d} & H^1(K,\mathrm{Ind}_K^L \mathbb{F}_2) \arrow{d}\\
M^\times \otimes \mathbb{F}_2 \arrow{r}{N_{M/L}} & L^\times \otimes \mathbb{F}_2.
\end{tikzcd}
\end{center}

\end{lemma}
\begin{proof}
Consider the following diagram: 
\begin{center}
\begin{tikzcd}
H^1\big(K,\mathrm{Ind}_K^{M}\mathbb{F}_2\big)\arrow{r}\arrow{d} & H^1\big(L, \mathrm{Ind}_{L}^{M}\mathbb{F}_2\big)\arrow{r}\arrow{d} & M^\times \otimes \mathbb{F}_2\arrow{d}{N_{M/L}} \\
H^1\big(K,\mathrm{Ind}_K^{L}\mathbb{F}_2\big)\arrow{r} & H^1\big(L, \mathbb{F}_2\big)\arrow{r} & L^\times \otimes \mathbb{F}_2,
\end{tikzcd}
\end{center}
where the horizontal maps are the isomorphisms from Shapiro's lemma and Kummer theory, the left vertical map is induced by the map in the statement of the lemma and the central vertical map is induced by the degree map
\begin{align*}
\mathrm{Ind}_L^M \mathbb{F}_2 &\longrightarrow \mathbb{F}_2 \\
\sigma \otimes_{\mathbb{Z}[G_M]} 1 & \longmapsto 1.
\end{align*}
The right-hand square commutes due to \cite[Theorem 1.1]{SCHAEFER1995219}, via the identification $\mathrm{Ind}_L^M\mathbb{F}_2\simeq \mu_2(M\otimes L^{\mathrm{sep}})$. It remains to prove the commutativity of the left-hand square. To this end, suppose that a class $[\xi]\in H^1(K,\mathrm{Ind}_K^M\mathbb{F}_2)$ is represented by the cocycle
\begin{align}\label{ExpCoc}
\begin{split}
\xi: G_K &\longrightarrow \mathrm{Ind}_K^M\mathbb{F}_2  \\
\sigma &\longmapsto \sum_i \sigma_i \otimes_{\mathbb{Z}[G_M]} 1,
\end{split}
\end{align}
for automorphisms $\sigma_i \in G_K$. To determine the image of $[\xi]$ in $H^1(L,\mathbb{F}_2)$ by traversing the square clockwise, we must consider the following composition: 
\begin{align*}
H^1(K,\mathrm{Ind}_K^M\mathbb{F}_2) \xrightarrow{\sim } H^1(K,\mathrm{Ind}_K^L\mathrm{Ind}_L^M\mathbb{F}_2) \xrightarrow{\sim} H^1(K,\mathrm{coInd}_K^L\mathrm{Ind}_L^M\mathbb{F}_2) \xrightarrow{\mathrm{sh}^1} H^1(L,\mathrm{Ind}_L^M\mathbb{F}_2),
\end{align*}
where the first map is the obvious one, the second is induced by the explicit isomorphism between induction and coinduction (see \cite[Proposition 5.9]{brown2012cohomology}), and the final map is the explicit isomorphism from Shapiro's lemma (see \cite[Proposition 8]{StixShapiro}). Applying each of these maps to $\xi$ above, we obtain that its image in $H^1(L,\mathrm{Ind}_L^M\mathbb{F}_2)$ is represented by the cocycle
\begin{align}\label{AlmostThere}
\begin{split}
G_L &\longrightarrow \mathrm{Ind}_L^M\mathbb{F}_2 \\
\sigma &\longmapsto \sum_{i | \sigma_i \in G_L} \sigma_i\otimes_{\mathbb{Z}[G_M]}1,
\end{split}
\end{align}
where the $\sigma_i\in G_K$ are defined by the image of $\sigma$ under $\xi$, as above. Finally, by the explicit isomorphisms mentioned above, it straightforwardly follows that the image of $[\xi]$ in $H^1(L,\mathbb{F}_2)$, when traversing the square counterclockwise, is represented by a cocycle that maps some $\sigma \in G_L$ to the degree of the corresponding sum in \eqref{AlmostThere}, as required. 
\end{proof}

Returning to the case of interest $K=\mathbb{Q}_2$, we may then exploit Lemma \ref{GeneralInd} to obtain the desired description of the map \eqref{FieldTheoreticMap} via the following definition. 

\begin{definition}\label{L2def}
On the level of finite $G_K$-sets, the map 
\begin{align}\label{Injection}
\begin{split}
\mathrm{Ind}_K^{K_Q}\mathbb{F}_2 &\longrightarrow \wedge^2 \mathrm{Ind}_K^{K_f}\mathbb{F}_2\\
\beta_i&\longmapsto \alpha_{i,1}\wedge \alpha_{i,2}. 
\end{split}
\end{align}
defines a $G_K$-equivariant injection, which induces a decomposition of $\wedge^2 \mathrm{Ind}_K^{K_f}\mathbb{F}_2$ as a direct sum. Define the finite extension $L^{(2)}/K_Q^{(2)}$ of \'{e}tale algebras such that the cokernel of \eqref{Injection} is isomorphic to the induced module $\mathrm{Ind}_K^{L^{(2)}}\mathbb{F}_2$.
\end{definition}

\begin{lemma}\label{ExplicitMap}
We have the commutative diagram
\begin{center}
\begin{tikzcd}
H^1\big(K,\wedge^2 \mathrm{Ind}_K^{K_f} {\mathbb{F}}_2 \big) \arrow{r}\arrow{d} & H^1\big(K,\wedge^2 \mathrm{Ind}_K^{K_Q} {\mathbb{F}}_2 \big)\arrow{d} \\
K_f^{(2),\times} \otimes {\mathbb{F}}_2 \arrow{r} & K_Q^{(2),\times} \otimes {\mathbb{F}}_2
\end{tikzcd}
\end{center}
where the vertical maps are the compositions of the isomorphisms from Shapiro's lemma and Kummer theory via Lemma \ref{BK1Lemmas}(1), the top horizontal map is induced by \eqref{IndReductionWedgeMap}, and the bottom horizontal map is the composition
\begin{align*}
K_f^{(2),\times} \otimes {\mathbb{F}}_2\xlongrightarrow\sim K_Q^\times \otimes \mathbb{F}_2\oplus L^{(2),\times}\otimes {\mathbb{F}}_2 \xlongrightarrow{} K_Q^{(2),\times} \otimes {\mathbb{F}}_2,
\end{align*}
where this final map is the projection to $L^{(2),\times}\otimes \mathbb{F}_2$ composed with the norm map down to $K_Q^{(2)}$. 
\end{lemma}
\begin{proof}
By Definition \ref{L2def}, we have the isomorphism
\begin{align}\label{L2isom}
K_f^{(2),\times}\otimes \mathbb{F}_2\simeq K_Q^\times \otimes \mathbb{F}_2 \oplus L^{(2),\times}\otimes \mathbb{F}_2,
\end{align}
where the direct summand $K_Q^\times \otimes \mathbb{F}_2$ is obtained from the image of \eqref{Injection}. As a result, the map
\begin{align*}
K_Q^\times \otimes \mathbb{F}_2 \rightarrow K_Q^{(2),\times}\otimes \mathbb{F}_2
\end{align*}
induced by the pairs of roots of $f$ that reduce to the same root of $Q$ is trivial. In contrast, the pairs of roots reducing to distinct roots of $Q$ define a map 
\begin{align}\label{4to1map}
\mathrm{Ind}_K^{L^{(2)}}\mathbb{F}_2 \longrightarrow \mathrm{Ind}_K^{K_Q^{(2)}}\mathbb{F}_2.
\end{align}
If $L^{(2)}$ and $K_Q^{(2)}$ are fields, then it is clear that the reduction map \eqref{4to1map} is precisely of the form \eqref{GeneralMap}, and the result follows from Lemma \ref{GeneralInd}. Otherwise, if we decompose the \'{e}tale algebras as products of fields
\begin{align*}
L^{(2)}=\prod_{i,j} L_{ij}, \tab K_Q^{(2)}=\prod_i K_i,
\end{align*}
then the result follows by applying Lemma \ref{GeneralInd} to the maps
\begin{align*}
\mathrm{Ind}_K^{L_{ij}}\mathbb{F}_2 \rightarrow \mathrm{Ind}_K^{K_i}\mathbb{F},
\end{align*}
for each field extension $L_{ij}/K_i$. 
\end{proof}

\begin{remark}\label{DiffModels}
In the case where $X$ has a RWP, defining the above map is subtle as Lemma \ref{BK1Lemmas} and the isomorphism \eqref{OrdinaryWedge} require different models for $X$. Specifically, it is essential in the former that $f$ is of odd degree and thus $\deg Q <g+1$, whereas the latter requires $\deg Q =g+1$ due to the assumption of ordinary reduction at $2$. This subtlety is addressed below in \eqref{Q2ord}.
\end{remark}

\section{Descent in the Chabauty-Kim method}\label{SectionChabauty}
In this brief section, we recall the setup of the Chabauty-Kim method, as developed in \cite{kim2005motivic} and \cite{kim2009unipotent}. We then discuss the utility of $2$-descent for Bloch-Kato Selmer groups, as introduced in \cite{bk1} and \cite{bk2}, in deducing explicit finiteness criteria of depth 2 Chabauty-Kim sets. Although the Chabauty-Kim method applies in greater generality, we assume that $X$ is a smooth projective curve over $\mathbb{Q}$ of genus $g \geq 2$. 

Fix a prime $p$ of good reduction for $X$ and suppose we have a rational point $b\in X(\mathbb{Q})$. Denote by $U_n\coloneqq U_n(b)$ the maximal $n$-unipotent quotient of the $\mathbb{Q}_p$-pro-unipotent \'{e}tale fundamental group $\pi_1^{\mathbb{Q}_p}(X_{\overline{\mathbb{Q}}};b)$ of $X_{\overline{\mathbb{Q}}}$ based at $b$; for the definition of this fundamental group, see \cite[§10]{MR1012168}. Additionally, for a rational point $x\in X(\mathbb{Q})$, denote by $U_n(b,x)$ the path torsor obtained by pushing out the $\mathbb{Q}_p$-pro-unipotent \'{e}tale path torsor $\pi_1^{\mathbb{Q}_p}(X_{\overline{\mathbb{Q}}};b,x)$ along the quotient map. Finally, let $T$ be a finite set of primes that contains the primes of bad reduction for $X$ and the fixed prime $p$.

Central to the Chabauty-Kim method is the following commutative diagram:
\begin{center}
\begin{tikzcd}
X(\mathbb{Q}) \arrow{r}\arrow{d}{\kappa_n} & X(\mathbb{Q}_p) \arrow{d}{\kappa_{n,p}}\\
H^1(G_{\mathbb{Q},T},U_n) \arrow{r}{\mathrm{loc}_{n,p}} & H^1(\mathbb{Q}_p,U_n),
\end{tikzcd}
\end{center}
where $\kappa_n$, the \textit{pro-unipotent Kummer map}, assigns to a rational point $x\in X(\mathbb{Q})$ the corresponding path torsor $U_n(b,x)$ in $H^1(G_{\mathbb{Q},T},U_n)$; for any prime $l$, denote by $\kappa_{n,l}$ the local analogue of $\kappa_n$, mapping to $H^1(\mathbb{Q}_l,U_n)$. By \cite[Proposition 2]{kim2005motivic}, the non-abelian cohomology schemes $H^1(G_{\mathbb{Q},T},U_n)$ and $H^1(\mathbb{Q}_p,U_n)$ are isomorphic to affine $\mathbb{Q}_p$-schemes, such that the localisation map is a morphism of schemes. Moreover, the local \textit{Bloch-Kato Selmer scheme} $H^1_f(\mathbb{Q}_p,U_n)$, which parametrises crystalline $G_{\mathbb{Q}_p}$-equivariant $U_n$-torsors, is isomorphic to a $\mathbb{Q}_p$-subvariety of $H^1(\mathbb{Q}_p,U_n)$. Although the corresponding global object of interest has various definitions, we follow the variant introduced in \cite{Balakrishnan_2018}: denote by $\mathrm{Sel}(U_n)$ the \textit{Selmer variety}, defined as the subscheme of $H^1(G_{\mathbb{Q},T},U_n)$ parametrising classes $\xi \in H^1(G_{\mathbb{Q},T},U_n)$ for which
\begin{itemize}
    \item[(1)] $\mathrm{loc}_{n,l}(\xi)\in \kappa_{n,l}(X(\mathbb{Q}_l)) \textrm{ for all } l\neq p$;
    \item[(2)] $ \mathrm{loc}_{n,p}(\xi) \in H^1_f(\mathbb{Q}_p,U_n)$;
    \item[(3)] $\xi$ is mapped into the image of $J(\mathbb{Q})\otimes \mathbb{Q}_p\rightarrow H^1(G_{\mathbb{Q},T},V_pJ)$,
\end{itemize}
where $V_pJ$ denotes the $\mathbb{Q}_p$-rational Tate module of the Jacobian of $X$. Finally, define the \textit{depth $n$ Chabauty-Kim set} $X(\mathbb{Q}_p)_n$ as the set of $\mathbb{Q}_p$-rational points of $X$ that are mapped by $\kappa_{n,p}$ into the image of the Selmer variety. 

We emphasise that this overview of the Chabauty-Kim method is intentionally concise, as the technical details of the setup will not be explicitly utilised in the subsequent material. More detailed expositions can be found in any of the references on the Chabauty-Kim method mentioned thus far. Nevertheless, with this introduced notation, the following essential theorem will provide the desired finiteness criteria. 

\begin{theorem}[\texorpdfstring{\cite[Theorem 1]{kim2009unipotent}}{[Kim09, Theorem 1]}]
If the image of $\mathrm{Sel}(U_n)$ under the morphism $\mathrm{loc}_{n,p}$ is not dense in $H^1_f\big(\mathbb{Q}_p,U_n\big)$, then $X(\mathbb{Q}_p)_n$ is finite. 
\end{theorem}

This is the general theorem of Kim alluded to in the introduction; we restate the corresponding case for $n=2$ for ease of reference.

\begin{corollary}[\texorpdfstring{\cite[Lemma 25]{bk1}}{[Dog23, Lemma 25]}] \label{netancor}
Suppose 
\begin{align}\label{DimIneq}
\dim_{\mathbb{Q}_p} H_f^1(\mathbb{Q},\wedge^2 V_pJ) < \frac{1}{2}(3g-2)(g+1)-\mathrm{rk}\hspace{0.8mm} J(\mathbb{Q}).
\end{align}
Then $X(\mathbb{Q}_p)_2$ is finite. 
\end{corollary}

This corollary guarantees that the \textit{dimension inequality} holds:
\begin{align*}
\dim_{\mathbb{Q}_p} \mathrm{Sel}(U_2)<\dim_{\mathbb{Q}_p} H^1_f(\mathbb{Q}_p,U_2),
\end{align*}
which suffices to conclude non-dominance of $\mathrm{loc}_{2,p}$. The dimensions of these Selmer schemes can be expressed in terms of the dimensions of Bloch-Kato Selmer groups associated to the graded pieces of the filtration of $U_2$ determined by the exact sequence 
\begin{align*}
0\longrightarrow \overline{\wedge^2 V_pJ}\longrightarrow U_2 \longrightarrow V_pJ \longrightarrow 0, 
\end{align*}
where $\overline{\wedge^2 V_pJ}\coloneqq \mathrm{Coker}\big(\mathbb{Q}_p(1)\xrightarrow{\cup^*} \wedge^2 V_pJ\big)$.
To establish the inequality in Corollary \ref{netancor}, it remains to compute the dimensions of the relevant Bloch-Kato Selmer groups, which is possible via well-known results in $p$-adic Hodge theory. 

As in the paper from which Corollary \ref{netancor} is taken, the strategy of the subsequent sections is to bound the dimension on the left-hand side of \eqref{DimIneq} by the $\mathbb{F}_p$-dimension of a related Galois cohomology group more amenable to computation. To this end, we require the following definition. 

\begin{definition}\label{BKSgroupT}
Suppose $K$ is a number field, $V$ is a $\mathbb{Q}_p$-vector space with a continuous action of $G_K$, and $T\subseteq V$ is a $G_K$-stable $\mathbb{Z}_p$-lattice. For a prime $v$ of $K$, denote by
\begin{align*}
\iota: H^1(K_v,T)\longrightarrow H^1(K_v,V)
\end{align*} 
the map induced by the inclusion $T\subseteq V$. Define $H^1_*(K_v,T)\coloneqq \iota^{-1}\big(H^1_*(K_v,V)\big)$, where $*\in\{f,g\}$. The global analogue $H^1_*(K,T)$ is then defined as 
\begin{align*}
H^1_*(K,T)\coloneqq \bigcap_v\mathrm{Ker}\big(H^1(K,T)\rightarrow H^1(K_v,T)/H^1_*(K_v,T)\big),
\end{align*}
where the intersection varies over the primes $v$ of $K$. Finally, define $H^1_*(K,T\otimes \mathbb{F}_p)$ as the image of $H^1_*(K,T)$ in $H^1(K,T\otimes \mathbb{F}_p)$. 
\end{definition}

Thus, to bound the dimension on the left-hand side of \eqref{DimIneq}, it suffices to bound the rank of $H^1_f(\mathbb{Q},\wedge^2 T_pJ)$, which, in turn, is bounded by the $\mathbb{F}_p$-dimension of $H^1_f(\mathbb{Q},\wedge^2 J[p])$. Hereafter, fix $p=2$, so that we may exploit the general results of Section \ref{SectionPrereqs} and results from explicit $2$-descent for Jacobians of hyperelliptic curves (as detailed in \cite{stoll2001implementing}, for example). A first upper bound for the desired $\mathbb{F}_2$-dimension is provided by the following lemma.

\begin{lemma}[ \texorpdfstring{\cite[Lemmas 17, 18]{bk1}}{[Dog23, Lemmas 17, 18]}]\label{unramified}
Suppose $S$ is a set of primes that contains $2$ and the primes where $X$ does not have semistable reduction. Then 
\begin{align}\label{SUnitBound}
H_f^1(\mathbb{Q},\wedge^2 J[2]) \subseteq H^1(G_S,\wedge^2 J[2]).
\end{align}
\end{lemma}

The cohomology group $H^1(G_S,\wedge^2 J[2])$ is effectively computable (at least in the case that $X$ has a RWP) given Lemma \ref{BK1Lemmas} and effective computability of the cohomology groups of the form $H^1(G_{K,S},\mathbb{F}_2)$, as discussed in \cite[Proposition 12.6]{poonsch}, for instance. Specifically, if $K$ is a number field, then there exists the short exact sequence
\begin{align}\label{unramifiedF2}
0\longrightarrow \mathcal{O}_{K,S}^\times \otimes \mathbb{F}_2 \longrightarrow H^1(G_{K,S},\mathbb{F}_2) \longrightarrow \mathrm{Cl}(\mathcal{O}_{K,S})[2] \longrightarrow 0,
\end{align}
from which a basis of $H^1(G_{K,S},\mathbb{F}_2)$ can be computed given that a computer algebra can determine a basis for both $\mathcal{O}_{K,S}^\times$ and $\mathrm{Cl}(\mathcal{O}_{K,S})[2]$. If, instead, $K=\prod K_i$ is a finite \'{e}tale $\mathbb{Q}$-algebra, the analogous sequence holds if we define 
\begin{align*}
\mathcal{O}_{K,S}\coloneqq \prod_i \mathcal{O}_{K_i,S}, \tab \mathrm{Cl}(\mathcal{O}_{K,S})\coloneqq \prod_i \mathrm{Cl}(\mathcal{O}_{K_i,S}).
\end{align*}
In the case that $X$ has a RWP, the effectively computable dimension of $H^1(G_S,\wedge^2J[2])$ can be used to deduce finiteness of $X(\mathbb{Q}_2)_2$, via Corollary \ref{netancor}, for particular curves of genus higher than $2$, as illustrated by the following example. 
\begin{example}
Consider the hyperelliptic curve $X$, given by the Weierstrass model \eqref{WeierstrassModel}, where 
\begin{align*}
P(x)=x^7 + 3x^6 - 6x^4 - x^3 + 4x^2 - x, \tab Q(x)=1.
\end{align*}
This curve is from the genus $3$ \href{https://math.mit.edu/~drew/gce_genus3_hyperelliptic.txt}{database} of hyperelliptic curves, obtained using the methods of \cite{Genus3}. Using \texttt{Magma} \cite{Magma}, we verify that $X$ has good reduction away from $9835601$, where it has semistable reduction, and $\mathrm{Cl}(\mathbb{Q}_f^{(2)})=1$ for $f=4P+Q^2$. Therefore, Lemma \ref{BK1Lemmas}(2) and the exact sequence \eqref{unramifiedF2} imply 
\begin{align}\label{FirstExample}
H^1(G_S,\wedge^2 J[2])\simeq \mathrm{Ker}\big(\mathcal{O}_{\mathbb{Q}_f^{(2)}}[\tfrac{1}{2}]^\times \otimes \mathbb{F}_2 \rightarrow \mathcal{O}_{\mathbb{Q}_f}[\tfrac{1}{2}]^\times \otimes \mathbb{F}_2\big),
\end{align}
where we have taken $S=\{2\}$. We compute the bases of the $2$-units in the right-hand side of \eqref{FirstExample}, and compute that $H^1(G_S,\wedge^2 J[2])$ has dimension $10$. Finally, the function \texttt{RankBounds} in \texttt{Magma} computed the rank of $J(\mathbb{Q})$ to be $3$, and therefore $X(\mathbb{Q}_2)_2$ is finite by Corollary \ref{netancor} and Lemma \ref{unramified}. 
\end{example}
In contrast to the above example, there are no examples on the LMFDB of a genus $2$ curve for which finiteness of $X(\mathbb{Q}_2)_2$ can be deduced solely from Lemma \ref{unramified}. Similarly, the genus $3$ database contains no curves without a RWP for which finiteness can already be established. This is likely due to the fact that $\mathbb{Q}_f^{(2)}$ is $28$ dimensional over $\mathbb{Q}$ in this case. However, it is possible to get a crude upper bound for the relevant Bloch-Kato Selmer from the following lemma. 

\begin{lemma}\label{EvenLemmaGeneral}
Suppose that $X$ is a hyperelliptic curve of genus $g$ with no RWP, and denote by $S$ a nonempty set of primes. Then 
\begin{align*}
\mathrm{rk}\hspace{0.8mm} H^1(G_S,\wedge^2 T_2J) \leq \mathrm{rk}\hspace{0.8mm} H^1(G_S,\wedge^2 T_2 J_\mathfrak{m})-g.
\end{align*}
\end{lemma}
\begin{proof}
Considering the $2^n$-torsion in the exact sequence \eqref{GenJDefSES}, taking inverse limits, and then taking wedge squares yields the exact sequence 
\begin{align}\label{WedgeSqRelation}
0\longrightarrow T_2J(1)(\chi_c) \longrightarrow \wedge^2T_2J_\mathfrak{m} \longrightarrow \wedge^2 T_2J\longrightarrow 0, 
\end{align}
where $T_2J(1)(\chi_c)\coloneqq T_2J\otimes \mathbb{Z}_2(1)(\chi_c)$, and $\chi_c$ is defined as in \eqref{GenJDefSES}. By the Weil conjectures, we have that $H^0(G_S,\wedge^2T_2J)=0$, so the long exact sequence in cohomology associated to \eqref{WedgeSqRelation} allows the rank of $H^1(G_S,\wedge^2 T_2J)$ to be bounded above by 
\begin{align*}
\mathrm{rk}\hspace{0.8mm} H^1(G_S,\wedge^2T_2J_\mathfrak{m})-\Big(\mathrm{rk}\hspace{0.8mm} H^1\big(G_S,T_2J(1)(\chi_c)\big)-\mathrm{rk}\hspace{0.8mm}H^2\big(G_S,T_2J(1)(\chi_c)\big)\Big). 
\end{align*}
Given that $\dim H^0\big(G_S,V_2J(1)(\chi_c)\big)=0$ while $\dim H^0\big(\mathbb{R},V_2J(1)(\chi_c)\big)=g$, the result follows from the global Euler-Poincar\'{e} characteristic formula (see \cite[Remark 1.2.4]{FPR}).
\end{proof}

\begin{example}
Consider the hyperelliptic curve $X$ with no RWP, defined by the polynomials 
\begin{align*}
P(x)=-4x^6 - 4x^5 + 4x^4 + 5x^3 - 2x^2 - x, \tab Q(x)=x^4 + x^2 + 1,
\end{align*}
obtained from the genus $3$ database mentioned in Example \ref{FirstExample}. We compute, using \texttt{Magma}, that $X$ has semistable reduction away from $2$, and $\mathrm{Cl}(\mathcal{O}_{\mathbb{Q}_f^{(2)}})=1$ for $f=4P+Q^2$.  Therefore, the exact sequence \eqref{unramifiedF2} implies that 
\begin{align*}
\mathcal{A}(f,S)=\mathrm{Ker}\big(\mathcal{O}_{\mathbb{Q}_f^{(2)}}[\tfrac{1}{2}]^\times \otimes \mathbb{F}_2 \rightarrow \mathcal{O}_{\mathbb{Q}_f}[\tfrac{1}{2}]^\times \otimes \mathbb{F}_2\big),
\end{align*}
where $\mathcal{A}(f,S)$ is as defined in the introduction, and $S=\{2\}$. We compute bases for the relevant $2$-units and deduce that the dimension of $\mathcal{A}(f,S)$ is $22$. By Lemma \ref{WedgeSqJm}, we have the inclusion $H^1(G_S,\wedge^2 J_\mathfrak{m}[2])\subseteq \mathcal{A}(f,S)$, from which we deduce that the rank of $H^1(G_S,\wedge^2 T_2J_\mathfrak{m})$ is at most $22$. Finally, by Lemmas \ref{unramified} and \ref{EvenLemmaGeneral} we conclude that the dimension of $H^1_f(\mathbb{Q},\wedge^2 V_2J)$ is at most $19$. In contrast, the dimension of $H^1_f(G_S,\wedge^2 V_2J)$, conditional on the Bloch-Kato conjectures, is $5$ (see Section \ref{BKSubsection} below). The \texttt{RankBounds} function could only guarantee that the rank of $J(\mathbb{Q})$ is either $3$ or $4$, which, in either case, implies that the bound obtained is insufficient for deducing finiteness of $X(\mathbb{Q}_2)_2$ via Corollary \ref{netancor}. In Example \ref{LastExample} below, we provide a similar example where finiteness can be deduced, given the refinements discussed in the remainder of the paper. 
\end{example}

To refine the inclusion in \eqref{SUnitBound}, the strategy is roughly to obtain obstructions for an element in $H^1(G_S,\wedge^2 J[2])$ to lie in $H^1_f(\mathbb{Q}_2, \wedge^2 J[2])$ upon restriction to $G_{\mathbb{Q}_2}$. This is precisely the strategy of \cite{bk1}, which constructs $H^1_f(\mathbb{Q}_2,\wedge^2 J[2])$ via a non-abelian analogue of the explicit Kummer map (referred to as the $(x-T)$ map in \cite{poonsch}) from $2$-descent for Jacobians of hyperelliptic curves. Unfortunately, this construction is restrictive and has only been used to verify finiteness of the depth $2$ set in one example. In this paper, we focus on hyperelliptic curves with good ordinary reduction at $2$ to exploit the following description.
\begin{lemma}\label{MainDesc}
Suppose that $X$ is a hyperelliptic curve over $\mathbb{Q}_2$ with good ordinary reduction. Then 
\begin{align*}
H^1_g(\mathbb{Q}_2,\wedge^2V_2J)= \mathrm{Ker} \big(H^1(\mathbb{Q}_2,\wedge^2V_2J)\rightarrow H^1(\mathbb{Q}_2,\wedge^2V_2J_{\mathbb{F}_2})\big).
\end{align*}
\end{lemma}
\begin{proof}
This is an instance of a general result (see \cite[§2.1.6(ii)]{NekPla}, for example) concerning Bloch-Kato Selmer groups of \textit{ordinary representations}. Regardless, we outline the proof here for completeness. 

The assumption of good ordinary reduction ensures that there exists the following short exact sequence of $\mathbb{Z}_2[G_{\mathbb{Q}_2}]$-modules:
\begin{align*}
0 \longrightarrow (V_2J_{\mathbb{F}_2})^*(1) \longrightarrow V_2J \longrightarrow V_2J_{\mathbb{F}_2} \longrightarrow 0,
\end{align*}
where the action of $G_{\mathbb{Q}_2}$ on $V_2J_{\mathbb{F}_2}$ is unramified. This induces a filtration $F^\bullet$ on $\wedge^2 V_2J$ by subrepresentations stable under the action of $G_{\mathbb{Q}_2}$, with graded pieces 
\begin{align*}
F^0/F^1 & = \wedge^2 V_2J_{\mathbb{F}_2}, \\ F^1/F^2  &= (V_2J_{\mathbb{F}_2}\otimes (V_2J_{\mathbb{F}_2})^*)(1), \\ F^2/F^3 &= (\wedge^2 V_2J_{\mathbb{F}_2})^*(2).
\end{align*}
If we denote by $W=F^1$ the kernel of $\wedge^2 V_2J\rightarrow \wedge^2 V_2J_{\mathbb{F}_2}$, it follows directly from the Galois action on the graded pieces above that
\begin{align*}
\mathrm{Fil}^0D_{\mathrm{dR}}(W) =D_{\mathrm{dR}}(\wedge^2 V_2J_{\mathbb{F}_2})/\mathrm{Fil}^0 = 0.
\end{align*} 
Consequently, the sequence 
\begin{align}\label{H1gSES}
H^1(\mathbb{Q}_2,W)\longrightarrow H^1_g(\mathbb{Q}_2, \wedge^2 V_2 J) \longrightarrow H^1_g(\mathbb{Q}_2,\wedge^2 V_2J_{\mathbb{F}_2})
\end{align}
is exact. However, the dimension of the right-most Selmer group can be deduced via the following series of equalities:
\begin{align*}
\dim_{\mathbb{Q}_2} H^1_g(\mathbb{Q}_2,\wedge^2 V_2J_{\mathbb{F}_2})=\dim_{\mathbb{Q}_2} D_{\mathrm{cris}}(\wedge^2 V_2J_{\mathbb{F}_2})^{\varphi =1} = \dim_{\mathbb{Q}_2} H^0(\mathbb{Q}_2,\wedge^2 V_2J_{\mathbb{F}_2}).
\end{align*}
Since $H^0(\mathbb{Q}_2,\wedge^2 V_2J_{\mathbb{F}_2})$ is trivial,  the exactness of \eqref{H1gSES} implies that
\begin{align*}
H^1_g(\mathbb{Q}_2,\wedge^2 V_2J)=\mathrm{Im}\big(H^1(\mathbb{Q}_2,W)\rightarrow H^1(\mathbb{Q}_2,\wedge^2 V_2J)\big),
\end{align*}
as desired.
\end{proof}

\begin{remark}\label{fDiffg}
The difference between $H^1_f(\mathbb{Q}_p,V)$ and $H^1_g(\mathbb{Q}_p,V)$, for a Galois $\mathbb{Q}_p$-representation $V$, is dictated by the following short exact sequence (see, for instance, \cite[Corollary 1.18]{Nekov2002OnPH}):
\begin{align}\label{DimDifffg}
0\longrightarrow H^1_f(\mathbb{Q}_p,V)\longrightarrow H^1_g(\mathbb{Q}_p,V)\longrightarrow \big(D_{\mathrm{cris}}(V^*(1))\big)^{\varphi =1}\longrightarrow 0.
\end{align}
Setting $p=2$ and $V=(V_2J)^{\otimes 2}$, this sequence, and the crystalline comparison theorem of Faltings \cite{FaltingsComparison}, imply that
\begin{align}\label{TensorSq}
\begin{split}
\dim_{\mathbb{Q}_p}H^1_g(\mathbb{Q}_2,( V_2 J)^{\otimes 2})-&\dim_{\mathbb{Q}_p}H^1_f(\mathbb{Q}_2,( V_2 J)^{\otimes 2})\\&= \dim \mathrm{End}_\varphi (H^1_{\mathrm{cris}}(X_{\mathbb{F}_2}/W(\mathbb{F}_2))[\tfrac{1}{2}]).    
\end{split}
\end{align}
Replacing $(V_2J)^{\otimes 2}$ with $\wedge^2 V_2J$, the difference in dimensions of the corresponding Selmer groups is the dimension of the alternating summand $\wedge^2 H^1_{\mathrm{cris}}(X_{\mathbb{F}_2},W(\mathbb{F}_2))[\tfrac{1}{2}](1)$ of the endomorphism group in \eqref{TensorSq}. This latter dimension can be determined via a crystalline analogue of Tate's isogeny theorem (see \cite[Theorems 5 and 6]{MilneWaterhouse}, for example). Indeed, the map 
\begin{align*}
\mathrm{End}(J_{\mathbb{F}_2})\otimes \mathbb{Z}_2\rightarrow  \mathrm{End}_\varphi (H^1_{\mathrm{cris}}(X_{\mathbb{F}_2}/W(\mathbb{F}_2))[\tfrac{1}{2}])
\end{align*}
is an isomorphism, under which the preimage of the above alternating summand is identified with the N\'{e}ron-Severi group of $J_{\mathbb{F}_2}$. In particular, 
\begin{align*}
\dim_{\mathbb{Q}_p}H^1_g(\mathbb{Q}_2,\wedge^2V_2 J)-&\dim_{\mathbb{Q}_p}H^1_f(\mathbb{Q}_2,\wedge^2V_2 J)=\rho(J_{\mathbb{F}_2})\geq 1,
\end{align*}
where $\rho(J_{\mathbb{F}_2})$ denotes the Picard number of $J_{\mathbb{F}_2}$.
\end{remark}

\section{Effective bounds and finiteness criteria}\label{SectionCriteria}
\subsection{Case 1: Hyperelliptic curves with a rational Weierstrass point}\label{Section1WP}
Throughout this subsection, we assume that $X$ is a hyperelliptic curve with precisely one RWP, given by the Weierstrass model \eqref{WeierstrassModel}. Applying the theory discussed in Section \ref{SectionPrereqs}, a finiteness criterion via Corollary \ref{netancor} follows straightforwardly. However, we can further strengthen this by including the methods of \cite{bk2}. As such, we first briefly outline the relevant theory from this paper. 

\subsubsection{Local obstructions at $2$ and $\infty$ via explicit boundary maps}\label{SSectionBoundary}
In this digression, we do not impose that $X$ has good ordinary reduction at $2$, though we certainly require that it has exactly one RWP. In \cite{bk2}, the inclusion in \eqref{SUnitBound} is refined by considering those classes in $H^1(G_S,\wedge^2 J[2])$ which lie in the image of the map 
\begin{align*}
H^1(\mathbb{Q}_v,\wedge^2 J[4]) \longrightarrow H^1(\mathbb{Q}_v,\wedge^2 J[2]),
\end{align*}
associated to the short exact sequence 
\begin{align}\label{BoundarySES}
0\longrightarrow \wedge^2 J[2] \longrightarrow \wedge^2 J[4] \longrightarrow \wedge^2 J[2] \longrightarrow 0,
\end{align}
for $v\in \{2,\infty\}$. This is achieved using the explicit field-theoretic description of $H^1(\mathbb{Q}_v,\wedge^2 J[2])$ in part (2) of Lemma \ref{BK1Lemmas}, along with explicitly describing the corresponding boundary map 
\begin{align}\label{BoundaryMapWedge}
H^1(\mathbb{Q}_v,\wedge^2 J[2]) \longrightarrow H^2(\mathbb{Q}_v,\wedge^2 J[2]),
\end{align}
associated to the sequence \eqref{BoundarySES}; this is given in \cite[Proposition 7]{bk2} as a map from the \'{e}tale algebra $\mathbb{Q}_{v,f}^{(2)}$. For $v\in\{2,\infty\}$, we denote by $\theta_v$ the composition 
\begin{align}\label{Thetas}
H^1(G_S,\wedge^2 J[2])\longrightarrow \mathrm{Ker}\big(\mathbb{Q}_{v,f}^{(2),\times} \otimes \mathbb{F}_2 \rightarrow \mathbb{Q}_{v,f}^\times \otimes \mathbb{F}_2\big) \longrightarrow \mathrm{Br}(\mathbb{Q}_{v,f}^{(2)})[2],
\end{align}
where this final map is the aforementioned explicit boundary map in \eqref{BoundaryMapWedge} via the field-theoretic description afforded by Lemma \ref{BK1Lemmas}. By construction, we have the inclusions 
\begin{align}\label{BoundaryInclusion}
H^1_f(\mathbb{Q},\wedge^2 J[2]) \subseteq \mathrm{Ker}\hspace{0.8mm}\theta_2 \cap \mathrm{Ker} \hspace{0.8mm} \theta_\infty \subseteq H^1(G_S,\wedge^2J[2]). 
\end{align}
The maps $\theta_2$ and $\theta_\infty$ are computable, and \cite[Theorem 1]{bk2} demonstrates the effectiveness of computing the intersection in \eqref{BoundaryInclusion} in deducing finiteness of $X(\mathbb{Q}_2)_2$ via Corollary \ref{netancor}. Specifically, it shows the desired finiteness can be established by this computation for the majority of genus $2$ odd hyperelliptic curves on the LMFDB with Mordell-Weil rank $2$.

\subsubsection{Statement of the main theorem in the odd degree case}
Suppose that $X$ has good ordinary reduction at $2$ and is given by the Weierstrass model \eqref{WeierstrassModel}. In the case of interest for this subsection, as noted in Remark \ref{DiffModels}, we require $\deg Q < g + 1$ to explicitly describe $H^1(\mathbb{Q},\wedge^2 J[2])$. To distinguish between the two cases considered in this paper, whenever $X$ is given by an odd degree model define $Q_{\beta}(x)=(x-\beta)^3Q\big(\tfrac{1}{x-\beta}\big)$, where $\beta \in \mathbb{Z}$ does not reduce to a root of $\overline{Q}\in \mathbb{F}_2[x]$. Then define 
\begin{align}\label{Q2ord}
\mathbb{Q}_2^{\mathrm{ord}}\coloneqq \begin{cases}
(\mathbb{Q}_2)_{Q}^{(2)}, &\text{if $\deg Q=g+1$}; \\[6pt]
(\mathbb{Q}_2)_{Q_{\beta}}^{(2)}, &\text{otherwise}.
\end{cases}
\end{align}
The effectively computable map essential in obtaining the effective bounds and finiteness criteria of this paper is then defined as follows. 

\begin{definition}\label{DefThetaDr}
Suppose we have the decomposition
\begin{align}\label{DecompQ2ord}
\mathbb{Q}_2^{\mathrm{ord},\times}\otimes \mathbb{F}_2\simeq \bigoplus_{i} K_i^\times \otimes \mathbb{F}_2,
\end{align}
where each $K_i$ is a finite extension of $\mathbb{Q}_2$. Denote by $\xi_i\in K_i^\times \otimes \mathbb{F}_2$ the class which defines the unramified quadratic extension of $K_i$. Then define 
\begin{align*}
\theta_{\mathrm{dR}}&:\mathrm{Ker}\Big(H^1\big(G_{\mathbb{Q}_f^{(2)},S},\mathbb{F}_2\big)\rightarrow H^1\big(G_{\mathbb{Q}_f,S},\mathbb{F}_2\big)\Big) \longrightarrow \bigoplus_i \frac{K_i^\times \otimes \mathbb{F}_2}{\langle \xi_i \rangle},
\end{align*}
as the map induced by the composition 
\begin{align*}
\mathbb{Q}_f^{(2),\times}\otimes \mathbb{F}_2 \rightarrow \mathbb{Q}_{2,f}^{(2),\times}\otimes \mathbb{F}_2\rightarrow \mathbb{Q}_2^{\mathrm{ord},\times}\otimes \mathbb{F}_2,
\end{align*}
where the final map is as defined in Lemma \ref{ExplicitMap}. 
\end{definition}

To finally relate $\theta_{\mathrm{dR}}$ to bounding the desired Selmer groups, we require the following lemma. It is then possible to state the main theorem for hyperelliptic curves with one RWP, before stating the corresponding finiteness criterion as an immediate corollary.

\begin{lemma}\label{unramtorsion}
The image of $2$-torsion under the composition
\begin{align}\label{unramMAp}
H^1(\mathbb{Q}_2,\wedge^2 T_2J _{\mathbb{F}_2}) \rightarrow H^1(\mathbb{Q}_2,\wedge^2 J_{\mathbb{F}_2}[2])\rightarrow \bigoplus_i K_i^\times \otimes \mathbb{F}_2
\end{align}
is the sum $\bigoplus_i\langle \xi_i \rangle $, where $\xi_i\in K_i^\times \otimes \mathbb{F}_2$ is defined as in Definition \ref{DefThetaDr}. 
\end{lemma}
\begin{proof}
Since $\wedge^2 J_{\mathbb{F}_2}[2]$ is isomorphic to the mod-$2$ quotient of $\wedge^2 T_2J_{\mathbb{F}_2}$ and $H^0(\mathbb{Q}_2,\wedge^2 T_2J_{\mathbb{F}_2})$ is trivial, it follows that the $2$-torsion in $H^1(\mathbb{Q}_2,\wedge^2 T_2J_{\mathbb{F}_2})$ is given by the image of the boundary map
\begin{align}\label{2torsioninJF2}
H^0(\mathbb{Q}_2,\wedge^2 J_{\mathbb{F}_2}[2])\rightarrow H^1(\mathbb{Q}_2,\wedge^2 T_2J_{\mathbb{F}_2}).
\end{align}
As $\wedge^2 T_2J_{\mathbb{F}_2}$ is unramified as a $G_{\mathbb{Q}_2}$-module, it follows that the image of \eqref{2torsioninJF2} is contained in 
\begin{align*}
H^1(\mathbb{F}_2,\wedge^2T_2J_{\mathbb{F}_2})=H^1_{\mathrm{nr}}(\mathbb{Q}_2,\wedge^2 T_2J_{\mathbb{F}_2}),
\end{align*}
from which we conclude that the image of $H^1(\mathbb{Q}_2,\wedge^2 T_2J_{\mathbb{F}_2})[2]$ under the first map in \eqref{unramMAp} is unramified. Moreover, the exact sequence \eqref{SpFibSES} implies that $H^1(\mathbb{F}_2,\wedge^2 J_{\mathbb{F}_2}[2])$ maps to 
\begin{align}\label{CopiesOfF2}
H^1(\mathbb{F}_2,\wedge^2\mathrm{Ind}_{\mathbb{Q}_2}^{\mathbb{Q}_{2,Q}}\mathbb{F}_2)\simeq \bigoplus_i \mathbb{F}_2,
\end{align}
where the index $i$ varies over the number of fields in the decomposition \eqref{DecompQ2ord} of $\mathbb{Q}_2^{\mathrm{ord}}$. The result then follows as the image of \eqref{CopiesOfF2} in $\mathbb{Q}_2^{\mathrm{ord},\times}\otimes \mathbb{F}_2$ is precisely the sum $\bigoplus_i\langle \xi_i \rangle $. 
\end{proof}

\begin{remark}
The cohomology group $H^0(\mathbb{Q}_2,\wedge^2 J_{\mathbb{F}_2}[2])$, and therefore the $2$-torsion in $H^1(\mathbb{Q}_2,\wedge^2 T_2J_{\mathbb{F}_2})$, is always nontrivial in the case that $X$ has ordinary reduction over $\mathbb{Q}_2$. This follows from the fact that the defining polynomial $Q$ being separable modulo $2$ implies that the Galois group of each irreducible factor of $Q$ over $\mathbb{Q}_2$ is cyclic. In particular, for an irreducible factor of $Q$ of degree $d$ with roots $\{\beta_{i_j}\}_{j=1}^d$, the sum
\begin{align*}
\beta_{i_1}\wedge \beta_{i_2}+\beta_{i_2}\wedge \beta_{i_3}+\cdots+\beta_{i_{d-1}}\wedge \beta_{i_d} + \beta_{i_1}\wedge \beta_{i_d}
\end{align*}
defines an element of $H^0(\mathbb{Q}_2,\wedge^2 J_{\mathbb{F}_2}[2])$.
\end{remark}

\begin{theorem}\label{MainOddTheorem}
Suppose that $X$ is a hyperelliptic curve of genus $g$ with exactly one RWP, given by the Weierstrass model \eqref{WeierstrassModel} such that $f= 4P+Q^2$ has odd degree. Let $S$ be a finite set of primes that contains $2$, where $X$ has good ordinary reduction, as well as the primes where $X$ does not have semistable reduction. Then 
\begin{align*}
\dim_{\mathbb{Q}_2} H^1_f(G_S,\wedge^2 V_2J)\leq \dim_{\mathbb{F}_2}\mathrm{Ker}(\theta_2\oplus \theta_\infty \oplus \theta_{\mathrm{dR}})-2.
\end{align*}
\end{theorem}
\begin{proof}
By Lemma \ref{MainDesc}, we obtain the inclusion
\begin{align*}
H^1_g(G_S,\wedge^2 J[2])\subseteq \mathrm{Ker}\big(H^1(G_S,\wedge^2 J[2])\rightarrow H^1(\mathbb{Q}_2,\wedge^2 J_{\mathbb{F}_2}[2])/\mathrm{Im}(H^1(\mathbb{Q}_2,\wedge^2T_2J_{\mathbb{F}_2})[2])\big),
\end{align*}
where the map on the right-hand side factors through $H^1(\mathbb{Q}_2,\wedge^2J[2])$. By Lemmas \ref{ExplicitMap} and \ref{unramtorsion}, and the isomorphism \eqref{OrdinaryWedge}, the right-hand side of the inclusion is contained in $\mathrm{Ker}(\theta_{\mathrm{dR}})$. 
By Remark \ref{fDiffg}, it then suffices to show that $ H^1_f(G_S, \wedge^2 T_2J) $ has nontrivial 2-torsion and that 
\begin{align}\label{ProveNeq}
    H_f^1(G_S, \wedge^2 J[2]) \cap \mathrm{Ker} (\theta_2 \oplus \theta_\infty) 
    \neq H^1_g(G_S, \wedge^2 J[2]) \cap \mathrm{Ker} (\theta_2 \oplus \theta_\infty),
\end{align}
from which the result follows from the inclusion \eqref{BoundaryInclusion}. Towards the former point, note that the Weil pairing, upon dualising, induces the injection 
\begin{align*}
    \mathbb{F}_2 \hookrightarrow \wedge^2 J[2].
\end{align*}
It follows that $ H^0(G_S, \wedge^2 J[2]) $ is nontrivial, and this group injects into $ H^1(G_S, \wedge^2 T_2J) $, as $ H^0(G_S, \wedge^2 T_2J) = 0 $. Finally, it is clear by definition that $ H^1_f(G_S, \wedge^2 T_2J) $ contains all the 2-torsion of $ H^1(G_S, \wedge^2 T_2J) $. 

Given the decomposition 
\begin{align*}
    \wedge^2 V_2J \simeq \mathbb{Q}_2(1) \oplus \overline{\wedge^2 V_2J},
\end{align*}
and the fact that $ H^1_f(G_S, \mathbb{Q}_2(1)) = 0 $ while $ H^1_g(G_S, \mathbb{Q}_2(1)) \neq 0 $, it follows that there exists an element in $H^1(G_S,\mathbb{Z}_2(1))$ whose image in $H^1_g(\mathbb{Q}_2,\wedge^2 J[2])$ under the composite 
\begin{align*}
H^1(G_S,\mathbb{Z}_2(1))\rightarrow H^1(G_S,\mu_2) \rightarrow H^1(\mathbb{Q}_2,\mu_2)\rightarrow H^1(\mathbb{Q}_2,\wedge^2 J[2])
\end{align*} 
is not contained in $H^1_f(\mathbb{Q}_2,\wedge^2 J[2])$. Finally, since $ H^1(\mathbb{Q}_v, \mu_2)$ in $ H^1(\mathbb{Q}_v, \wedge^2 J[2]) $ clearly lifts to $H^1(\mathbb{Q}_v, \mu_4) $ in $ H^1(\mathbb{Q}_v, \wedge^2 J[4])$ for $ v \in \{2, \infty\} $, this proves \eqref{ProveNeq}.
\end{proof}

\begin{corollary}\label{OddFiniteCriterion}
Suppose $X$ and $S$ are as in Theorem \ref{MainOddTheorem}. Then the depth $2$ set $X(\mathbb{Q}_2)_2$ is finite whenever 
\begin{align*}
\dim_{\mathbb{F}_2} \mathrm{Ker}(\theta_2 \oplus \theta_\infty \oplus \theta_{\mathrm{dR}})< \frac{1}{2}(3g^2+g+2)-\mathrm{rk}\hspace{0.8mm} J(\mathbb{Q}).
\end{align*}
\end{corollary}
\begin{proof}
This immediately follows from Theorem \ref{MainOddTheorem} and Corollary \ref{netancor}.
\end{proof}

\subsection{Case 2: Hyperelliptic curves without a rational Weierstrass point}

As discussed in Section \ref{SectionPrereqs}, contrary to the previous section, we no longer have an explicit field-theoretic description of $H^1(\mathbb{Q},\wedge^2 J[2])$ to obtain an upper bound for the rank of the relevant Bloch-Kato Selmer group. However, Lemma \ref{WedgeSqJm} provides a description of the analogue for the generalised Jacobian, which will suffice for deriving the required upper bounds. To that end, suppose that $X$ is a hyperelliptic curve of genus $g$ with no RWP, and $S$ is a nonempty set of primes. In the notation of Lemma \ref{EvenLemmaGeneral}, it is clear that 
\begin{align*}
H^1_g\big(G_S,V_2J(1)(\chi_c)\big)=H^1\big(G_S,V_2J(1)(\chi_c)\big),
\end{align*}
where $H^1_g(G_S,-)$ denotes those classes that are unramified outside of $S$ and de Rham at $2$. This immediately yields the following refinement of the inequality from Lemma \ref{EvenLemmaGeneral}:
\begin{align}\label{GoodEvenIneq}
\mathrm{rk}\hspace{0.8mm} H^1_g(G_S,\wedge^2 T_2J) \leq \mathrm{rk}\hspace{0.8mm} H^1_g(G_S,\wedge^2 T_2 J_\mathfrak{m})-g.
\end{align}
This inequality allows the relevant dimension bounds to be deduced solely from cohomology groups associated to the generalised Jacobian, which, in this case, are more amenable to computation. 

\begin{theorem}\label{MainEvenThm}
Suppose $X$ is a hyperelliptic curve over $\mathbb{Q}$ of genus $g$ without a RWP, given by the Weierstrass model \eqref{WeierstrassModel}. Suppose $S$ is a finite set of primes that contains $2$, where $X$ has good ordinary reduction, as well as all primes where $X$ does not have semistable reduction. Then 
\begin{align*}
\dim _{\mathbb{Q}_2} H^1(G_S,\wedge^2 V_2J)\leq\dim_{\mathbb{F}_2} \mathrm{Ker}(\theta_{\mathrm{dR}})-g-1
\end{align*}
\end{theorem}
\begin{proof}
By Lemma \ref{WedgeSqJm} and the short exact sequence \eqref{DefinitionEvenWedge}, there is an inclusion 
\begin{align*}
H^1(G_S,\wedge^2 J_\mathfrak{m}[2])\subseteq \mathrm{Ker}\Big(H^1\big(G_{\mathbb{Q}_f^{(2)},S},\mathbb{F}_2\big)\rightarrow H^1\big(G_{\mathbb{Q}_f,S},\mathbb{F}_2\big)\Big),
\end{align*} 
which gives the bound $\dim H^1_g(G_S,\wedge^2 T_2J_\mathfrak{m})\leq \dim \mathrm{Ker} (\theta_\mathrm{dR})$. The result follows from \eqref{GoodEvenIneq} and Remark \ref{fDiffg}.
\end{proof}

\begin{corollary}\label{EvenFiniteCriterion}
Suppose $X$ and $S$ are as in Theorem \ref{MainEvenThm}. Then the depth $2$ set $X(\mathbb{Q}_2)_2$ is finite whenever 
\begin{align*}
\dim_{\mathbb{F}_2} \mathrm{Ker}(\theta_{\mathrm{dR}})<\frac{3}{2}g(g+1)-\mathrm{rk}\hspace{0.8mm} J(\mathbb{Q}). 
\end{align*}
\end{corollary}

\begin{remark}\label{NoTorsRemark}
Contrary to Theorem \ref{MainOddTheorem}, the even degree case does not as readily allow for the refinements provided by $2$-torsion in $H^1(G_S, \wedge^2 T_2J_\mathfrak{m})$ or by explicit boundary maps. To establish nontrivial $2$-torsion as in the odd degree case, one would need to identify nontrivial elements of $H^0(G_S, \wedge^2 J_\mathfrak{m}[2])$, but it is not immediately clear how to deduce their existence. Moreover, it is unclear how to extend the boundary map approach from \cite{bk2} to the even degree case. For instance, the odd degree case relies on an explicit description of $H^1(K, \mathrm{End}_{\mathbb{F}_2}J[2])$, an analogue of which for the generalised Jacobian is subtler to describe. Although obstructions at the real place may be introduced following \cite[§5]{bk2}, proving the analogue of \eqref{ProveNeq} is also not immediate, as the injection $\mathbb{Q}_2(1) \to \wedge^2 V_2J$, crucial in the odd case, may no longer be applicable here.
\end{remark}

\section{Low genus examples}\label{SectionExamples}
Verifying the criteria above requires constructing the \'{e}tale $\mathbb{Q}_2$-algebra $L^{(2)}$ in Lemma \ref{ExplicitMap}, the structure of which is dictated by the Galois group of the defining polynomial $f$ of $X$ over $\mathbb{Q}_2$. Conveniently, the assumption of good ordinary reduction greatly limits the possible Galois groups of $f$, allowing for the straightforward implementation of an algorithm that checks the finiteness criteria for each possible construction of $L^{(2)}$ for low genus curves. The steps in this construction, which are independent of whether $X$ has a RWP and are suitable for defining the map in Lemma \ref{ExplicitMap}, are detailed below. 
\begin{algorithm}[Explicitly computing the map from Lemma \ref{ExplicitMap}(2)]\label{Algo} $ $

\noindent \textbf{Input:} Polynomials $P,Q\in \mathbb{Q}[X]$ which define a Weierstrass model \eqref{WeierstrassModel} for a hyperelliptic curve $X$ with good ordinary reduction at $2$, such that the degree of $f=4P+Q^2$ is odd if $X$ has a RWP.

\noindent \textbf{Output:} The \'{e}tale algebra $L^{(2)}$ and the map 
\begin{align*}
L^{(2),\times}\otimes \mathbb{F}_2 \rightarrow \mathbb{Q}_2^{\mathrm{ord},\times}\otimes \mathbb{F}_2.
\end{align*}

\begin{itemize}
    \item[(1)] Determine the polynomial $F\in \mathbb{Q}[x]$ for which $\mathbb{Q}_{f}^{(2)}\simeq \mathbb{Q}[x]/(F)$, and factor it over $\mathbb{Q}_2$. Each factor defines a field in the decomposition of $\mathbb{Q}_{2,f}^{(2)}$, which, in turn, is defined by a $G_{\mathbb{Q}_2}$-stable subset of $\wedge^2 \mathrm{Ind}_{\mathbb{Q}_2}^{\mathbb{Q}_{2,f}}\mathbb{F}_2$ upon taking cohomology.
    \item[(2)] Compute the Galois group of $f$ over $\mathbb{Q}_2$ and use this to match the factors of $F$ obtained in (1) with the $G_{\mathbb{Q}_2}$-orbits of the unordered pairs of distinct roots of $f$. Let each factor $F_i$ of $F$ correspond to an orbit denoted $G_{\mathbb{Q}_2}\cdot s_i$, where $s_i $ is contained in $\wedge^2 \mathrm{Ind}_{\mathbb{Q}_2}^{\mathbb{Q}_{2,f}}\mathbb{F}_2$.
    \item[(3)] Using the correspondence in (2), match each factor $F_i$ with the field $K_i$ in the decomposition $\mathbb{Q}_{2}^{\mathrm{ord},\times}\otimes \mathbb{F}_2=\bigoplus_i K_i^\times \otimes \mathbb{F}_2$ such that the image of $G_{\mathbb{Q}_2}\cdot s_i$ under the map in Lemma \ref{ExplicitMap} gives $K_i^\times \otimes \mathbb{F}_2$ upon taking cohomology.
    \item[(4)] Any factor of $F_i$ over the corresponding field $K_i$ obtained in (3) defines an extension of $K_i$, which is the relevant component of $L^{(2)}$.
    \item[(5)] Return the map 
    \begin{align*}
    L^{(2),\times}\otimes \mathbb{F}_2 \simeq \bigoplus_{i} \Big(\bigoplus_j (K_i)_{G_{ij}}^\times \otimes \mathbb{F}_2\Big) \xrightarrow{\bigoplus_i\mathrm{Nm}} \bigoplus_i K_i^\times \otimes \mathbb{F}_2, 
    \end{align*}
    where $j$ runs over all the factors $F_{ij}$ of $F$ that correspond to $K_i$ in (4), and $G_{ij}$ is any factor of $F_{ij}$ over $K_i$.
\end{itemize}
\end{algorithm}
To exploit this algorithm to define $\theta_{\mathrm{dR}}$, it would remain to compute the unramified classes $\xi_i$ from Definition \ref{DefThetaDr} in each of the fields $K_i$ obtained in step (3). In the remainder of this section, we present examples computing $\theta_\mathrm{dR}$ associated to genus $2$ and $3$ curves by implementing the general framework outlined in Algorithm \ref{Algo}. All computations necessary in deducing these examples were implemented in \texttt{Magma} \cite{Magma}.

\subsection{Conditional dimension formulae}\label{BKSubsection}

Although the computations below provide an upper bound for the dimension of $H^1_f(G_S,\wedge^2 V_2J)$, the exact dimension is predicted by the Bloch-Kato conjectures, as previously mentioned. We now state the relevant part of the conjectures required for the conditional dimension calculations, followed by a lemma instrumental in the application of this conjecture.
\begin{conjecture}[\texorpdfstring{\cite[Conjecture 5.3(i)]{bloch1990functions}}{[BK90, Conjecture 5.3(i)]}]\label{BKConjectures}
Suppose $A$ is a smooth projective variety over $\mathbb{Q}$. Then the \'{e}tale regulator 
\begin{align*}
\mathrm{CH}^i(A,j)\otimes \mathbb{Q}_p \rightarrow H^1_f\big(\mathbb{Q},H^{2i-j-1}_{\textrm{\'{e}t}}(A_{\overline{\mathbb{Q}}},\mathbb{Q}_p(i))\big),
\end{align*}
is an isomorphism whenever $2i-j-1>0$.
\end{conjecture}

\begin{lemma}[\texorpdfstring{\cite[Lemma 2.4]{QC2}}{[BD20, Lemma 2.4]}]\label{QC2Lemma}
Suppose $X$ is a smooth projective curve over $\mathbb{Q}$, and $S$ is a finite set of primes. Then Conjecture \ref{BKConjectures} implies 
\begin{align*}
H^1_f(G_S,\overline{\wedge^2V_2J}^*(1))=0.
\end{align*}
\end{lemma}

Following the proof of \cite[Lemma 2.5]{QC2}, if $S$ is a finite set of primes, then we obtain the conditional dimension 
\begin{align}\label{BKDim}
\begin{split}
\dim_{\mathbb{Q}_p} H^1_f &(G_S,\wedge^2 V_2J) \\&=\dim_{\mathbb{Q}_p}H^1_f(\mathbb{Q}_p,\overline{\wedge^2 V_2J})-g(g-1)-\mathrm{rk}\hspace{0.8mm} \mathrm{NS}(J) + 1 \\
&=\frac{1}{2}(g^2+g-2)-\mathrm{rk}\hspace{0.8mm} \mathrm{NS}(J)+1,
\end{split}
\end{align}
where the first equality is a consequence of Lemma \ref{QC2Lemma} and Poitou-Tate duality, while the second follows from the equality of dimensions $\dim H^1_f(\mathbb{Q}_p,\overline{\wedge^2 V_2J}) = \dim H^1_f(\mathbb{Q}_p,\wedge^2 V_2J)-1$ and the isomorphism $H^1_f(\mathbb{Q}_p,\wedge^2 V_2J)\simeq D_{\mathrm{dR}}(\wedge^2V_2J)/\mathrm{Fil}^0$.

Under the assumption that $X$ is a hyperelliptic curve with good ordinary reduction at $2$, Conjecture \ref{BKConjectures} further predicts a greater difference in the dimensions of the global Bloch-Kato Selmer groups of Remark \ref{fDiffg}, as discussed in the following lemma.

\begin{lemma}\label{ConditionalDiff}
Suppose $X$ is a hyperelliptic curve of genus $g$ over $\mathbb{Q}$ with good ordinary reduction at $2$, and $S$ is a finite set of primes that contains $2$. Then Conjecture \ref{BKConjectures} implies that 
\begin{align*}
\dim_{\mathbb{Q}_p}H^1_g(G_S,\wedge^2 V_2J)-\dim _{\mathbb{Q}_p} H^1_f(G_S,\wedge^2 V_2J) \geq g. 
\end{align*}
\end{lemma}
\begin{proof}
Given the filtration of $\wedge^2 V_2J$ in the proof of Lemma \eqref{MainDesc}, in the case that $X$ has good ordinary reduction at $2$ we know that the difference in dimensions between the local Selmer groups $H^1_*(\mathbb{Q}_2,\wedge^2 V_2J)$ for $*\in \{f,g\}$ is precisely the difference in dimensions of the corresponding Selmer groups of $\mathrm{gr}^1F^\bullet = V_2J_{\mathbb{F}_2}^*\otimes V_2J_{\mathbb{F}_2}(1)$. By the exact sequence \eqref{DimDifffg}, this dimension is the following:
\begin{align*}
\dim \big(D_{\mathrm{cris}}\big((\mathrm{gr}^1F^\bullet)^*(1)\big)\big)^{\varphi =1}&= \dim D_{\mathrm{cris}}(V_2J_{\mathbb{F}_2}^*\otimes V_2J_{\mathbb{F}_2})^{\varphi = 1} \\
&= \dim \mathrm{End}_\varphi \big(D_{\mathrm{cris}}(V_2J_{\mathbb{F}_2})\big).
\end{align*}
Unconditionally, this final dimension is at least $g$. Indeed, if $V_2J_{\mathbb{F}_2}$ decomposes as the direct sum $\oplus_i V_i$ of simple representations, then $\mathrm{End}_\varphi \big(D_{\mathrm{cris}}(V_i)\big)$ contains the subspace $\langle 1, \varphi, \varphi^2,\dots,\varphi^{\dim V_i-1}\rangle$. Summing over all $V_i$ gives the desired $g$-dimensional subspace of $\mathrm{End}_\varphi \big(D_{\mathrm{cris}}(V_2J_{\mathbb{F}_2})\big)$.

This disparity in the dimensions of the local Selmer groups lifts to the global Selmer groups. To that end, suppose to the contrary that 
\begin{align*}
H^1_g(G_S,\wedge^2V_2J)\rightarrow H^1_g(\mathbb{Q}_2,\wedge^2V_2J)/H^1_f(\mathbb{Q}_2,\wedge^2V_2J)
\end{align*}
does not surject. This implies that the image of 
\begin{align*}
H^1_f\big(G_S,(\wedge^2V_2J)^*(1)\big)\rightarrow H^1_f\big(\mathbb{Q}_2,(\wedge^2V_2J)^*(1)\big)/H^1_e\big(\mathbb{Q}_2,(\wedge^2V_2J)^*(1)\big)
\end{align*}
is nontrivial. Indeed, this follows as the exact annihilators of $H^1_g(\mathbb{Q}_p,W)$ and $H^1_f(\mathbb{Q}_p,W)$, for a $p$-adic Galois representation $W$, under the cup product pairing
\begin{align*}
H^1(\mathbb{Q}_p,W)\times H^1(\mathbb{Q}_p,W^*(1))\rightarrow \mathbb{Q}_p,
\end{align*}
are $H^1_e(\mathbb{Q}_p,W^*(1))$ and $H^1_f(\mathbb{Q}_p,W^*(1))$, respectively (see \cite[Proposition 3.8]{bloch1990functions}). However, this nontrivial image contradicts Lemma \ref{QC2Lemma}.
\end{proof}
\begin{remark}\label{ConditionalDiffRemark}
Given Lemma \ref{ConditionalDiff} and Remark \ref{fDiffg}, we expect the bounds obtained using Theorems \ref{MainOddTheorem} and \ref{MainEvenThm} to, at best, be $g-1$ above the dimension predicted by the Bloch-Kato conjectures in \eqref{BKDim}.   
\end{remark}

\subsection{Case 1: Curves with exactly one RWP}
\subsubsection{Genus 2}
We first consider genus $2$ hyperelliptic curves given by the Weierstrass model \eqref{WeierstrassModel}, where $\overline{Q}(x)=x^2+x+1$ in $\mathbb{F}_2[x]$. We restrict our attention to this case as there are only four possible Galois groups for $f$ over $\mathbb{Q}_2$. Additionally, the vast majority of relevant genus $2$ hyperelliptic curves listed on the LMFDB \cite{lmfdb}, which have good ordinary reduction at $2$ and exactly one rational Weierstrass point, can be straightforwardly represented by such a model. 

As $f$ is of odd degree, it is necessary to modify the pairing of its roots in Notation \ref{RootsNotation}. Specifically, fix $\alpha_1$ as the $\mathbb{Q}_2$-rational root of $f$, which reduces to the rational root $\beta_1$ of $Q_{\beta}$ in \eqref{Q2ord}, and retain the usual notation for the remaining roots (i.e., the root $\alpha_{i,j}$ of $f$ reduces to the roots $\beta_i$ of $Q$ for $i\in\{2,3\}$, $j\in\{1,2\}$). It follows from basic Galois theory that the possible Galois groups for $f$ over $\mathbb{Q}_2$ are isomorphic to the following subgroups of $S_5$: (i) $\langle (24)(35) \rangle$, (ii) $\langle(24)(35),(23)(45)\rangle$, (iii) $\langle (2435)\rangle$ and (iv) $\langle (24)(35), (23) \rangle$. A straightforward computation then determines the $G_{{\mathbb{Q}}_2}$-orbits of the pairs of roots of $f$, and hence the factorisation of the defining polynomial $F$ of $\mathbb{Q}_{2,f}^{(2)}$ over $\mathbb{Q}_2$. We illustrate this procedure for the largest Galois group, case (iv), which is generic for the curves considered from the LMFDB; with obvious modifications, the other cases are treated analogously. 

\begin{example}\label{OddGenus2}
Consider the hyperelliptic curve $X$ with LMFDB label \href{https://www.lmfdb.org/Genus2Curve/Q/216663/a/216663/1}{216663.a.216663.1}, given by the Weierstrass model \eqref{WeierstrassModel}, where 
\begin{align*}
P(x)&=x^5 - 2x^4 + 3x^3 - 4x^2+x-1, \tab Q(x)=x^2 + x + 1.
\end{align*}
This curve has Mordell-Weil rank 2 and discriminant $3\cdot 72221$, so that it is possible to take $S=\{2\}$. Additionally, $\mathrm{Cl}(\mathbb{Q}_f^{(2)})=\mathrm{Cl}(\mathbb{Q}_f)=1$ for 
\begin{align*}
f(x)&=4x^5 - 7x^4 + 14x^3 - 13x^2 + 6x - 3,
\end{align*}
and therefore 
\begin{align}\label{sunits}
H^1(G_S,\wedge^2 J[2]) \simeq \mathrm{Ker}\big(\mathcal{O}_{\mathbb{Q}_f^{(2)}}[\tfrac{1}{2}]^\times \otimes \mathbb{F}_2 \rightarrow \mathcal{O}_{\mathbb{Q}_f}[\tfrac{1}{2}]^\times \otimes \mathbb{F}_2 \big),
\end{align}
by Lemma \ref{BK1Lemmas} and the exact sequence \eqref{unramifiedF2}. According to \texttt{Magma}, the dimension of \eqref{sunits} is $6$. It is readily verified that $\mathrm{Gal}(f)\simeq \langle (24)(35), (23) \rangle$, and in this case $F$ splits over $\mathbb{Q}_2$ as the product of a quadratic and two quartic factors. Explicitly, the quadratic factor corresponds to the pairs of roots of $f$ that reduce to the same root of $Q$, as expected, while the quartic factors arise from the following orbits of pairs of roots of $f$:
\begin{itemize}
    \item[(i)] $G_{\mathbb{Q}_2}\cdot \{\alpha_1,\alpha_{2,1}\} = \big\{\{\alpha_1,\alpha_{2,1}\},\{\alpha_1,\alpha_{2,2}\},\{\alpha_1,\alpha_{3,1}\},\{\alpha_1,\alpha_{3,2}\}\big\}$; 
    \item[(ii)] $G_{\mathbb{Q}_2}\cdot \{\alpha_{2,1},\alpha_{3,1}\}=\big\{\{\alpha_{2,1},\alpha_{3,1}\},\{\alpha_{2,1},\alpha_{3,2}\},\{\alpha_{2,2},\alpha_{3,1}\},\{\alpha_{2,2},\alpha_{3,2}\}\big\}$.
\end{itemize}
To distinguish between the two quartic factors in the \texttt{Magma} implementation, note that the factor corresponding to (i) will factor over $\mathbb{Q}_{2,Q}$ as the product of two irreducible quadratics, while the factor corresponding to (ii) will remain irreducible. For $K=\mathbb{Q}_2$, under the map
\begin{align*}
\wedge^2\mathrm{Ind}_{K}^{K_f}\mathbb{F}_2\rightarrow \wedge^2 \mathrm{Ind}_{K}^{K_{Q_{\beta}}}\mathbb{F}_2,
\end{align*}
the orbit in (i) is mapped to the $G_{\mathbb{Q}_2}$-orbit $\big\{\{\beta_1,\beta_2\},\{\beta_1,\beta_3\}\big\}$, which defines a degree $2$ unramified extension of $\mathbb{Q}_2$ upon taking Galois cohomology. The orbit in (ii) maps to $\big\{\{\beta_2,\beta_3\}\big\}$, defining a copy of $\mathbb{Q}_2$ passing to cohomology. In particular, this implies that 
\begin{align*}
\mathbb{Q}_2^{\mathrm{ord},\times}\otimes \mathbb{F}_2\simeq \mathbb{Q}_{2,Q}^\times \otimes \mathbb{F}_2 \oplus \mathbb{Q}_2^\times \otimes \mathbb{F}_2.
\end{align*}
Consequently, denoting the quartic factors corresponding to (i) and (ii) as $G$ and $H$ respectively, $G$ factors as the product of two quadratics over $\mathbb{Q}_{2,Q}$. We then conclude that $\theta_{\mathrm{dR}}$ is described via the following map:
\begin{align}\label{example1}
L^{(2),\times} \otimes \mathbb{F}_2 \simeq  (K_Q)_{G_1}^\times \otimes \mathbb{F}_2 \oplus K_{H}^\times \otimes \mathbb{F}_2 \xlongrightarrow{ \mathrm{Nm}} K_Q^\times \otimes \mathbb{F}_2\oplus K^\times \otimes \mathbb{F}_2,
\end{align}
where, to avoid awkward notation, we again denote $K=\mathbb{Q}_2$, and $G_1$ is either of the quadratic factors of $G$ over $K_Q$. To use this description, it remains to determine the kernel of the map, factoring through \eqref{example1}, from the obtained explicit basis of \eqref{sunits} to the quotient of $ K_Q^\times \otimes \mathbb{F}_2 \oplus K^\times \otimes \mathbb{F}_2$ by the appropriate unramified classes. \texttt{Magma} verifies that this kernel has dimension $5$, from which we can conclude that $X(\mathbb{Q}_2)_2$ is finite by Corollary \ref{OddFiniteCriterion}. Note that by Remark \ref{ConditionalDiffRemark}, an upper bound of $5$ is the optimal bound we expect from the methods of this paper. 
\end{example}

It is a straightforward computation to show that the remaining $3$ possibilities for the Galois group of $f$ over $\mathbb{Q}_2$ give $2$ additional factorisations of $F$: either (i) $F$ factors as the product of $3$ quadratic and one quartic factors, or (ii) as the product of $2$ linear and $4$ quadratic factors. The procedure for these cases is analogous to the example above. Moreover, it is computationally inexpensive to check this criterion for a given curve, allowing for the automation of this verification process for the appropriate genus $2$ curves on the LMFDB.
\begin{theorem}\label{ExampleComputation}
Of the $6,603$ genus $2$ curves with Mordell-Weil rank $2$ and exactly one rational Weierstrass point on the LMFDB, $1,138$ have good ordinary reduction at $2$, and at least $574$ of these satisfy $\#X(\mathbb{Q}_2)_2<\infty$.

\end{theorem}

From the above curves that meet the criterion of Corollary \ref{OddFiniteCriterion}, the finiteness of $X(\mathbb{Q}_2)_2$ can be verified without the boundary map considerations of Section \ref{SSectionBoundary} for $487$ curves. An example of a rank $2$ curve for which finiteness of $X(\mathbb{Q}_2)_2$ cannot be deduced solely through determining $\mathrm{Ker}(\theta_{\mathrm{dR}})$ or using the boundary map considerations of \cite{bk2}, but can be verified by combining both methods, is the hyperelliptic curve with LMFDB label \href{https://www.lmfdb.org/Genus2Curve/Q/44071/a/44071/1}{44071.a.44071.1}, with Weierstrass model \eqref{WeierstrassModel} where 
\begin{align*}
P(x)=x^5-8x^3-8x^2+6x+7, \tab Q(x)=x^2+x+1.
\end{align*}

\begin{remark}
As mentioned above, one advantage of Corollary \ref{OddFiniteCriterion} over the explicit boundary map computations of \cite{bk2} is that computing the kernel of $\theta_2\oplus \theta_\infty$ is significantly more computationally expensive than computing the kernel of $\theta_{\mathrm{dR}}$, as defined in Theorem \ref{MainOddTheorem}. For instance, when using \texttt{Magma} to compute these conditions for the $1,138$ curves mentioned in Theorem \ref{ExampleComputation}, it took $46.638$ seconds per curve to check $\dim \mathrm{Ker}(\theta_2 \oplus \theta _\infty)<6$, whereas checking the corresponding condition for $\theta_{\mathrm{dR}}$ took an average of $2.020$ seconds per curve.
\end{remark}

\subsubsection{Genus 3}
In the genus $3$ case, it is no longer feasible to amass examples as in Theorem \ref{ExampleComputation} for genus 2. Although there is the aforementioned database of genus $3$ curves, it is too computationally expensive to check the criterion for every relevant curve. Indeed, the \'{e}tale algebra $\mathbb{Q}_f^{(2)}$ in this case has degree $21$ over $\mathbb{Q}$, and computing its (in the best case) $2$-units and $2$-class group is highly inefficient. Moreover, unlike the convenience afforded by the LMFDB for genus $2$ curves, this database does not provide the Mordell-Weil rank of the Jacobian of each curve, further increasing the computational time for finding appropriate examples. Nevertheless, it is not difficult to find examples within this database, and we include an example for both cases of interest. 

\begin{example}
Consider the hyperelliptic curve $X$, obtained from the aforementioned database of genus $3$ curves, given by the Weierstrass model \eqref{WeierstrassModel}, where 
\begin{align*}
P(x)=x^7 + 3x^6 - 3x^5 - 7x^4 + 4x^2 + 2x, \tab Q(x)=x^3 + x^2 + 1.
\end{align*}
The \texttt{Magma} function \texttt{RankBounds} verifies that the Mordell-Weil rank of the Jacobian of this curve is $3$. Furthermore, $X$ has semistable reduction away from $2$, and $\mathrm{Cl}(\mathbb{Q}_f^{(2)})=1$ for 
\begin{align*}
f(x)=4x^7 + 13x^6 - 10x^5 - 27x^4 + 2x^3 + 18x^2 + 8x + 1.
\end{align*} 
Therefore, we similarly obtain the description \eqref{sunits} as above, verifying that its dimension is $14$. As such, we cannot yet conclude finiteness of $X(\mathbb{Q}_2)_2$ as in Example \ref{FirstExample}, and need to proceed with computing $\mathrm{Ker}(\theta_{\mathrm{dR}})$. In contrast to Example \ref{OddGenus2}, the Galois group of $f$ over $\mathbb{Q}_2$ is not maximal: it is isomorphic to the subgroup of $S_7$ generated by $(23)(45)(67)$ and $(246)(357)$. Given that the $G_{\mathbb{Q}_2}$-orbits of $\{\beta_1,\beta_2\}$ and $\{\beta_2,\beta_3\}$ have order $3$, we obtain the isomorphism 
\begin{align}\label{q2ordex}
\mathbb{Q}_2^{\mathrm{ord},\times}\otimes \mathbb{F}_2\simeq (\mathbb{Q}_{2,Q}^\times\otimes \mathbb{F}_2)^{\oplus 2}.
\end{align}
It is readily checked that $F$ splits over $\mathbb{Q}_2$ as the product of a cubic and $3$ sextic factors. The cubic is as expected, and the sextic factors are determined by the $G_{\mathbb{Q}_2}$ orbits of $\{\alpha_1,\alpha_{2,1}\}$, $\{\alpha_{2,1},\alpha_{3,1}\}$ and $\{\alpha_{2,1},\alpha_{3,2}\}$, using analogous labeling of the roots as in Example \ref{OddGenus2}. Clearly, the first orbit maps to a $G_{\mathbb{Q}_2}$-set isomorphic to $\mathrm{Ind}_K^{K_Q}\mathbb{F}_2$ corresponding to a component of \eqref{q2ordex}; the second and third orbits both map to a set corresponding to the remaining component. Let $F=G\cdot F_1\cdot F_2 \cdot F_3$ over $\mathbb{Q}_2$, where $G$ is the cubic factor and $F_i$ denotes the sextic factor that corresponds to the $i^\mathrm{th}$ orbit above. Therefore, to construct $\theta_{\mathrm{dR}}$, we consider the following map:
\begin{align*}
(K_Q)_{F_{11}}^\times \otimes \mathbb{F}_2 \oplus \big((K_Q)_{F_{21}}^\times \otimes \mathbb{F}_2 \oplus (K_Q)_{F_{31}}^\times \otimes \mathbb{F}_2\big) \xrightarrow{\mathrm{Nm}} K_Q^\times \otimes \mathbb{F}_2 \oplus K_Q^\times \otimes \mathbb{F}_2,
\end{align*}
where $K=\mathbb{Q}_2$, and $F_{i,1}$ denotes any of the quadratic factors of $F_i$ over $K_Q$. To implement this map in \texttt{Magma}, it is necessary to distinguish the sextic factor $F_1$ from $F_2$ and $F_3$. One method is to fix a root of $f$ that is not in $\mathbb{Q}_2$ and compute which $F_i$ has a root whose difference from $\alpha$ is contained in $\mathbb{Q}_2$. As before, we can use this map and the obtained description of \eqref{sunits} to compute that the dimension of the kernel of $\theta_{\mathrm{dR}}$ is $11$. It follows from Corollary \ref{OddFiniteCriterion} that $X(\mathbb{Q}_2)_2$ is finite. 
\end{example}

\subsection{Case 2: Curves with no RWP}
\subsubsection{Genus 2}
Although there are genus $3$ examples, none of the genus $2$ hyperelliptic curves on the LMFDB with no RWP and good ordinary reduction at $2$ satisfy the criterion of Corollary \ref{EvenFiniteCriterion}. In contrast, \cite[Lemma 2.6]{QC2} showed that Lemma \ref{QC2Lemma} implies that finiteness of $X(\mathbb{Q}_2)_2$ follows if $\mathrm{rk}\hspace{0.8mm}J(\mathbb{Q})<g^2$, conditional on Conjecture \ref{BKConjectures}. As mentioned above, one possible explanation for this disparity could be explained by $2$-torsion in $H^1(G_S,\wedge^2T_2J_\mathfrak{m})$, although, as discussed in Remark \ref{NoTorsRemark}, proving the existence of such torsion cannot be deduced following the strategy in the odd degree case. In forthcoming work, we attempt to reconcile this disparity by instead considering explicit methods to effectively describe the difference in dimensions between the global Selmer groups, motivated by Remark \ref{ConditionalDiffRemark}. Nevertheless, the current methods do provide new bounds for the Bloch-Kato Selmer groups of interest; we illustrate this procedure with an example, before detailing some statistics on the bounds obtained for the curves with no RWP on the LMFDB. 

\begin{example}\label{Genus2noRWPExample}
Consider the hyperelliptic curve $X$ with LMFDB label \href{https://www.lmfdb.org/Genus2Curve/Q/10651/a/10651/1}{10651.a.10651.1}, given by the Weierstrass model \eqref{WeierstrassModel}, where 
\begin{align*}
P(x)=-x^5 - 2x^4 - x, \tab Q(x)=x^3+x+1.
\end{align*}
As above, $X$ has Mordell-Weil rank $2$ and semistable reduction away from $2$. However, in this example the $2$-class group of $\mathbb{Q}_f^{(2)}$ is isomorphic to $\mathbb{Z}/2\mathbb{Z}$. Nevertheless, we proceed as above and use \texttt{Magma} to compute a basis of
\begin{align*}
\mathrm{Ker}\Big(H^1\big(G_{\mathbb{Q}_f^{(2)},\{2\}},\mathbb{F}_2\big)\rightarrow H^1\big(G_{\mathbb{Q}_f,\{2\}},\mathbb{F}_2\big) \Big),
\end{align*}
via the sequence \eqref{unramifiedF2}, and verify it is of dimension $8$. The Galois group of $f$ over $\mathbb{Q}_2$ is isomorphic to the subgroup of $S_6$ generated by $(1,3,5)(2,4,6)$ and $(1,2)$, and $F$ factors as the product of a cubic and a degree $12$ factor. The latter factor arises from the $G_{\mathbb{Q}_2}$-orbit of the pairs of roots of $f$ which do not reduce to the same root of $Q$, and therefore $L^{(2)}$ is just a quartic field extension of $\mathbb{Q}_{2,Q}$. If $G_1$ denotes any of the quartic factors over $K_Q$ of the degree $12$ factor of $F$, then
\begin{align*}
L^{(2),\times}\otimes \mathbb{F}_2 \simeq (K_Q)_{G_1}^\times \otimes \mathbb{F}_2 \xrightarrow{\mathrm{Nm}} K_Q^\times \otimes \mathbb{F}_2. 
\end{align*}
As above, we can use this to show that the dimension of $\mathrm{Ker}(\theta_{\mathrm{dR}})$ is $7$. By Corollary \ref{EvenFiniteCriterion}, finiteness of $X(\mathbb{Q}_2)_2$ can only be concluded if this dimension if less than 7.  
\end{example}

\begin{remark}
This example is illustrative of the fact that the bounds obtained in this case are not far from producing a wide-reaching finiteness criteria for genus $2$ hyperelliptic curves with no RWP. To elaborate, suppose that, for computational ease, we consider only curves given by the Weierstrass model \eqref{WeierstrassModel} where $\overline{Q}(x)$ is either $x^3+x^2+1$ or $x^3+x+1$ in $\mathbb{F}_2[x]$. In these cases, there are three possible factorisations for the defining polynomial $F$ of $\mathbb{Q}_{2,f}^{(2)}$ over $\mathbb{Q}_2$: it factors as the product of either (i) five cubic factors; (ii) a cubic and two sextic factors; or (iii) a cubic and a degree 12 factor, as in Example \ref{Genus2noRWPExample}. Automating the procedure for computing $\mathrm{Ker}(\theta_{\mathrm{dR}})$, \texttt{Magma} verifies that of the $13,337$ genus $2$ curves with Mordell-Weil rank $2$ and no RWP on the LMFDB, $5,246$ have good ordinary reduction at $2$, and at least $1,492$ of these satisfy $\dim \mathrm{Ker}(\theta_{\mathrm{dR}})=7$. 
\end{remark}

\subsubsection{Genus 3}
As above, the genus $3$ case does not afford sufficient computational efficiency to check the finiteness criterion for a large number of curves. Regardless, the following example can be found on the aforementioned database of genus $3$ hyperelliptic curves. 
\begin{example}\label{LastExample}
Consider the genus $3$ hyperelliptic curve $X$, given by the Weierstrass model \eqref{WeierstrassModel}, where 
\begin{align*}
    P(x)=-4x^6 - 7x^5 + 4x^4 + 14x^3 + 5x^2 - 2x, \tab Q(x)=x^4 + x + 1.
\end{align*}
The relevant class groups are trivial, $X$ has semistable reduction away from $2$, and
\begin{align*}
\mathrm{Ker}\big(\mathcal{O}_{\mathbb{Q}_f^{(2)}}[\tfrac{1}{2}]^\times\otimes\mathbb{F}_2\rightarrow\mathcal{O}_{\mathbb{Q}_f}[\tfrac{1}{2}]^\times\otimes\mathbb{F}_2 \big)
\end{align*}
has dimension $15$. The \texttt{Magma} function \texttt{RankBounds} can only verify that the Mordell-Weil rank of its Jacobian is $3$ or $4$, however, this will be sufficient to deduce finiteness of $X(\mathbb{Q}_2)_2$. The Galois group of $f$ over $\mathbb{Q}_2$ is isomorphic to the subgroup of $S_8$ generated by $(1,3,5,7)(2,4,6,8)$ and $(1,2)$, so that $F$ factors as the product of a quartic, a degree $8$ factor and a degree $16$ factor over $\mathbb{Q}_2$. The degree $8$ factor arises from the orbit $G_{\mathbb{Q}_2}\cdot \{\alpha_{11},\alpha_{31}\}$; the degree $16$ factor from $G_{\mathbb{Q}_2}\cdot \{\alpha_{11},\alpha_{21}\}$. The component of $L^{(2)}$ corresponding to the first orbit (resp. the second orbit) maps to the degree $2$ (resp. degree $4$) unramified extension $K_2^{nr}$ (resp. $K_4^{nr}$) of $\mathbb{Q}_{2}$. Denote by $G$ and $H$ the degree $8$ and $16$ factors, respectively. Then we are interested in the map 
\begin{align*}
(K_2^{nr})_{G_1}^\times \otimes \mathbb{F}_2 \oplus (K_4^{nr})_{H_1}^\times \otimes \mathbb{F}_2 \xrightarrow{\mathrm{Nm}}K_2^{nr,\times}\otimes \mathbb{F}_2 \oplus K_4^{nr,\times}\otimes \mathbb{F}_2,
\end{align*}
where $K=\mathbb{Q}_2$ and $G_1$ (resp. $H_1$) denotes any of the factors of $G$ (resp. $H$) over $K_2^{nr}$ (resp. $K_4^{nr}$). Through this, \texttt{Magma} computed the dimension of the kernel of $\theta_{\mathrm{dR}}$ to be $13$, from which we can conclude that $X(\mathbb{Q}_2)_2$ is finite by Corollary \ref{EvenFiniteCriterion}. 
\end{example}

\bibliographystyle{alpha}
\bibliography{references}

\end{document}

%% file: header.tex
\usepackage{amsmath,amsfonts,amssymb,amsthm,enumitem,tikz-cd,graphicx,color,verbatim,xspace,bbm}
\usepackage[utf8]{inputenc}
\usepackage{amsmath,amsfonts,amssymb}
\usepackage{amsthm}
\usepackage{color,multicol}
\usepackage{yfonts}
\usepackage{tikz}
\usepackage[new]{old-arrows}
\usepackage{tikz-cd}
\usetikzlibrary{positioning, shapes.geometric, matrix}
\usetikzlibrary{matrix,decorations.pathreplacing, calc, positioning,fit}
\usepackage{scrextend}
\usepackage{adjustbox}
\usepackage{relsize}
\usepackage{fancyhdr}
\usepackage{graphicx}
\usepackage{babel,blindtext}
\usepackage{caption}
\usepackage{mathrsfs} 
\usepackage{scalerel}
\usepackage{mdframed}
\usepackage{centernot}
\usepackage{mathtools}
\usepackage{ stmaryrd }
\usepackage{enumitem}

\numberwithin{equation}{section}
\theoremstyle{plain}
\newcounter{theoremcount}[section]
\makeatletter
\@addtoreset{theoremcount}{section}
\@addtoreset{equation}{section}
\@addtoreset{equation}{section}
\makeatother

\newtheorem{theorem}[theoremcount]{Theorem}
\newtheorem*{theorem*}{Theorem}
\newtheorem{conjecture}[theoremcount]{Conjecture}
\newtheorem{lemma}[theoremcount]{Lemma}

\newtheorem*{proposition*}{Proposition}
\newtheorem{corollary}[theoremcount]{Corollary}
\theoremstyle{definition}
\newtheorem{definition}[theoremcount]{Definition}
\newtheorem{definition-construction}[theoremcount]{Definition/Construction}
\newtheorem{lemma-construction}[theoremcount]{Lemma/Construction}
\newtheorem{example}[theoremcount]{Example}
\newtheorem{remark}[theoremcount]{Remark}
\newtheorem{notation}[theoremcount]{Notation}

\newtheorem{algorithm}[theoremcount]{Algorithm}

\usepackage[OT2,T1]{fontenc}
\DeclareSymbolFont{cyrletters}{OT2}{wncyr}{m}{n}
\DeclareMathSymbol{\Sha}{\mathalpha}{cyrletters}{"58}

\newtheoremstyle{introstyle}% <name>
  {3pt}% <space above>
  {3pt}% <space below>
  {\itshape}% <body font>
  {}% <indent amount>
  {\bfseries}% <theorem head font>
  {.}% <punctuation after theorem head>
  {.5em}% <space after theorem head>
  {}% <theorem head spec>

\theoremstyle{introstyle}
\newtheorem{introtheorem}{Theorem}

\newtheorem*{nonumtheorem}{Theorem}

\usepackage[colorlinks=true, 
            backref=page, % Backreferences with page numbers
            citecolor=blue, 
            linkcolor=blue, 
            urlcolor=blue, 
            bookmarksopen=true]{hyperref}

%% file: main.bbl
\newcommand{\etalchar}[1]{$^{#1}$}
\begin{thebibliography}{{LMF}24}

\bibitem[BCP97]{Magma}
Wieb Bosma, John Cannon, and Catherine Playoust.
\newblock The {M}agma algebra system. {I}. {T}he user language.
\newblock {\em J. Symbolic Comput.}, 24(3-4):235--265, 1997.
\newblock Computational algebra and number theory (London, 1993).

\bibitem[BD18]{Balakrishnan_2018}
Jennifer~S. Balakrishnan and Netan Dogra.
\newblock {Quadratic Chabauty and rational points, I: $p$-adic heights}.
\newblock {\em Duke Mathematical Journal}, 167(11), August 2018.

\bibitem[BD20]{QC2}
Jennifer~S Balakrishnan and Netan Dogra.
\newblock {Quadratic Chabauty and Rational Points II: Generalised Height Functions on Selmer Varieties}.
\newblock {\em International Mathematics Research Notices}, 2021(15):11923--12008, 02 2020.

\bibitem[BDM{\etalchar{+}}23]{Balakrishnan_Dogra_Müller_Tuitman_Vonk_2023}
Jennifer~S. Balakrishnan, Netan Dogra, J.~Steffen Müller, Jan Tuitman, and Jan Vonk.
\newblock Quadratic chabauty for modular curves: algorithms and examples.
\newblock {\em Compositio Mathematica}, 159(6):1111–1152, 2023.

\bibitem[BK90]{bloch1990functions}
Spencer Bloch and Kazuya Kato.
\newblock L-functions and tamagawa numbers of motives.
\newblock In {\em The Grothendieck Festschrift: A Collection of Articles Written in Honor of the 60th Birthday of Alexander Grothendieck}, pages 333--400. Springer, 1990.

\bibitem[Blo86]{bloch1986algebraic}
Spencer Bloch.
\newblock {Algebraic cycles and higher $K$-theory}.
\newblock {\em Advances in mathematics}, 61(3):267--304, 1986.

\bibitem[Bro12]{brown2012cohomology}
Kenneth~S. Brown.
\newblock {\em {Cohomology of Groups}}.
\newblock Graduate Texts in Mathematics. Springer New York, 2012.

\bibitem[BS10]{Bruin_Stoll_2010}
Nils Bruin and Michael Stoll.
\newblock {The Mordell–Weil sieve: proving non-existence of rational points on curves}.
\newblock {\em LMS Journal of Computation and Mathematics}, 13:272–306, 2010.

\bibitem[BSS{\etalchar{+}}16]{Genus3}
Andrew~R. Booker, Jeroen Sijsling, Andrew~V. Sutherland, John Voight, and Dan Yasaki.
\newblock A database of genus-2 curves over the rational numbers.
\newblock {\em LMS J. Comput. Math.}, 19:235--254, 2016.

\bibitem[Cha41]{Chab}
Claude Chabauty.
\newblock {Sur les points rationnels des courbes algébriques de genre supérieur à l’unité}.
\newblock {\em C. R. Acad. Sci. Paris}, 212:882–885, 1941.

\bibitem[Col85]{colemanint}
Robert~F. Coleman.
\newblock Effective chabauty.
\newblock {\em Duke Math. J.}, 52(3):765–770, 1985.

\bibitem[Del89]{MR1012168}
Pierre Deligne.
\newblock Le groupe fondamental de la droite projective moins trois points.
\newblock In {\em Galois groups over {$\mathbb {Q}$} ({B}erkeley, {CA}, 1987)}, volume~16 of {\em Math. Sci. Res. Inst. Publ.}, pages 79--297. Springer, New York, 1989.

\bibitem[DM23]{DOKCHITSER2023264}
Vladimir Dokchitser and Adam Morgan.
\newblock A note on hyperelliptic curves with ordinary reduction over 2-adic fields.
\newblock {\em Journal of Number Theory}, 244:264--278, 2023.

\bibitem[Dog23]{bk1}
Netan Dogra.
\newblock {2-descent for Bloch--Kato Selmer groups and rational points on hyperelliptic curves I}.
\newblock {\em preprint}, 2023.

\bibitem[Dog24]{bk2}
Netan Dogra.
\newblock {2-descent for Bloch--Kato Selmer groups and rational points on hyperelliptic curves II}.
\newblock {\em preprint}, 2024.

\bibitem[Fal88]{FaltingsComparison}
Gerd Faltings.
\newblock {$p$-Adic Hodge Theory}.
\newblock {\em Journal of the American Mathematical Society}, 1(1):255--299, 1988.

\bibitem[FPR94]{FPR}
Jean-Marc Fontaine and Bernadette Perrin-Riou.
\newblock {Autour des conjectures de Bloch et Kato: cohomologie galoisienne et valeurs de fonctions L}.
\newblock In {\em Motives (Seatlle, WA, 1991)}, volume~55 of {\em Proc. Sympos. Pure Math.}, pages 599--706. Amer. Math. Soc., Providence, RI, 1994.

\bibitem[Kim05]{kim2005motivic}
Minhyong Kim.
\newblock {The motivic fundamental group of the projective line minus three points and the theorem of Siegel}.
\newblock {\em Inventiones mathematicae}, 161(3), 2005.

\bibitem[Kim09]{kim2009unipotent}
Minhyong Kim.
\newblock {The unipotent Albanese map and Selmer varieties for curves}.
\newblock {\em Publications of the Research Institute for Mathematical Sciences}, 45(1):89--133, 2009.

\bibitem[{LMF}24]{lmfdb}
The {LMFDB Collaboration}.
\newblock The {L}-functions and modular forms database.
\newblock \url{https://www.lmfdb.org}, 2024.

\bibitem[MW71]{MilneWaterhouse}
James~S. Milne and William~C. Waterhouse.
\newblock {Abelian varieties over finite fields}.
\newblock In {\em 1969 Number Theory Institute (Stony Brook)}, volume~20 of {\em Proc. Sympos. Pure Math.}, pages 53--64. Amer. Math. Soc., Providence, RI, 1971.

\bibitem[Nek93]{Nekov2002OnPH}
Jan Nekov{\'a}ř.
\newblock On $p$-adic height pairings.
\newblock In {\em Séminaire de Théorie des Nombres}, 1993.

\bibitem[NP00]{NekPla}
Jan Nekov{\'a}ř and Andrew Plater.
\newblock {On the parity of ranks of Selmer groups}.
\newblock {\em Asian J. Math.}, 4(2):437 -- 498, 2000.

\bibitem[PS97]{poonsch}
Bjorn Poonen and Edward Schaefer.
\newblock {Explicit descent for Jacobians of cyclic covers of the projective line}.
\newblock {\em J. Reine Angew. Math.}, 488:141--188, 1997.

\bibitem[Sch95]{SCHAEFER1995219}
Edward Schaefer.
\newblock {2-Descent on the Jacobians of Hyperelliptic Curves}.
\newblock {\em Journal of Number Theory}, 51(2):219--232, 1995.

\bibitem[Ser12]{serre2012algebraic}
Jean-Pierre Serre.
\newblock {\em Algebraic Groups and Class Fields}.
\newblock Graduate Texts in Mathematics. Springer New York, 2012.

\bibitem[Sti10]{StixShapiro}
Jakob Stix.
\newblock {Trading degree for dimension in the section conjecture: The non-abelian Shapiro lemma}.
\newblock {\em Mathematical Journal of Okayama University}, 52:29--43, 2010.

\bibitem[Sto01]{stoll2001implementing}
Michael Stoll.
\newblock {Implementing 2-descent for Jacobians of hyperelliptic curves}.
\newblock {\em Acta Arithmetica}, 98(3):245--277, 2001.

\end{thebibliography}
